\documentclass{amsart}
\usepackage{amssymb,amsthm,comment,color,enumerate,adjustbox,enumitem}
\numberwithin{equation}{section}

\usepackage{hyperref}
\newtheorem{theorem}{Theorem}[section]
\newtheorem{lemma}[theorem]{Lemma}
\newtheorem{proposition}[theorem]{Proposition}

\newtheorem{corollary}[theorem]{Corollary}

\def\cent#1#2{{\bf C}_{#1}(#2)}
\def\nor#1#2{{\bf N}_{#1}(#2)}
\def\Z#1{{\bf Z}(#1)}

\newcommand{\Aut}{\mathop{\mathrm{Aut}}}

\newcommand{\Sym}{\mathop{\mathrm{Sym}}}
\newcommand{\Alt}{\mathop{\mathrm{Alt}}}
\newcommand{\PSL}{\mathop{\mathrm{PSL}}}
\newcommand{\PGL}{\mathop{\mathrm{PGL}}}

\newcommand{\GL}{\mathop{\mathrm{GL}}}
\newcommand{\Inndiag}{\mathop{\mathrm{Inndiag}}}
\newcommand{\PSU}{\mathop{\mathrm{PSU}}}
\newcommand{\PGU}{\mathop{\mathrm{PGU}}}
\newcommand{\GU}{\mathop{\mathrm{GU}}}
\newcommand{\GF}{\mathop{\mathrm{GF}}}

\newcommand{\GammaO}{\mathop{\Gamma\mathrm{O}}}
\newcommand{\SU}{\mathop{\mathrm{SU}}}
\newcommand{\PSp}{\mathop{\mathrm{PSp}}}

\newcommand{\SO}{\mathop{\mathrm{SO}}}
\newcommand{\GO}{\mathop{\mathrm{GO}}}
\newcommand{\Sp}{\mathop{\mathrm{Sp}}}
\newcommand{\SL}{\mathop{\mathrm{SL}}}

\newcommand{\PGammaL}{\mathop{\mathrm{P}\Gamma\mathrm{L}}}
\newcommand{\POmega}{\mathop{\mathrm{P}\Omega}}

\definecolor{darkblue}{rgb}{0,0,0.8}


\begin{document}

\title[Szep's Conjecture]{A generalization of Szep's conjecture for almost simple groups}

\author{Nick Gill}
\address{ Department of Mathematics, The Open University, Milton Keynes MK7 6AA, U.K.}
\email{nick.gill@open.ac.uk}

\author{Michael Giudici}
\address{Michael Giudici, Centre for Mathematics of Symmetry and Computation,\newline
School of Mathematics and Statistics,
The University of Western Australia,
 Crawley, WA 6009, Australia}
\email{michael.giudici@uwa.edu.au}
\author{Pablo Spiga}
\address{Dipartimento di Matematica e Applicazioni, University of Milano-Bicocca, Via Cozzi 55, 20125 Milano, Italy} 
\email{pablo.spiga@unimib.it}

\begin{abstract}
We prove a natural generalization of Szep's conjecture. Given an almost simple group $G$ with socle not isomorphic to an orthogonal group having Witt defect zero, we classify all possible group elements $x,y\in G\setminus\{1\}$ with $G=\nor G {\langle x\rangle}\nor G{\langle y\rangle}$, where we are denoting by $\nor G{\langle x\rangle}$ and by $\nor G{\langle y\rangle}$ the normalizers of the cyclic subgroups $\langle x\rangle$ and $\langle y\rangle$. As a consequence of this result, we classify all possible group elements $x,y\in G\setminus\{1\}$ with $G=\cent G x\cent G y$.

\begin{center}
\textit{Dedicated to Pham Huu Tiep on the occasion of his $60^{\textrm{th}}$ birthday.}
\end{center}
\end{abstract}
\maketitle
\section{introduction}\label{intro}

Given a finite group $G$ and $x\in G$, we denote by $\cent G x$ the \emph{\textbf{centralizer}} of $x$ in $G$ and by $\nor G {\langle x\rangle}$ the \emph{\textbf{normalizer}} of the cyclic subgroup $\langle x\rangle$ in $G$. It was conjectured by J.~Szep~\cite{FA}, that if $G=\cent G x\cent G y$ with $x,y\in G\setminus\{1\}$, then $G$ is not a non-abelian simple group. Over a long period, many authors investigated this conjecture and in $1987$, using the Classification of the Finite Simple Groups, E.~Fisman and Z.~Arad~\cite[Theorem~1]{FA} gave a positive answer to this problem.

More recently, R.~Guralnick, G.~Malle and P.~Tiep have obtained another proof of Szep's conjecture~\cite{GMT} as a direct application of some results on the product of conjugacy classes in algebraic groups. This new proof, for Lie type groups, actually proves more than the original statement of Szep's conjecture. Namely, it is shown that, if $L$ is a non-abelian simple group of Lie type and $L\unlhd G\leqslant \Inndiag(L)$, then $\cent G x\cent G y \neq  G$, for every $x,y\in G\setminus\{1\}$. Here $\Inndiag(L)$ denotes the group of \textbf{{\em inner-diagonal}} automorphisms of $L$, as defined in~\cite[Chapter~$2$]{GLS}.

Moreover, recently Szep's conjecture has played a crucial role in the investigation of the finite primitive groups having two coprime subdegrees~\cite{nostro}.  Indeed, the positive solution of Szep's conjecture is used in Theorems~$1.5$ and~$1.6$ of~\cite{nostro}. In order to simplify some of the arguments in the proofs of these theorems, it would have been necessary to have Szep's conjecture available for the whole class of the finite almost simple groups.

The following is the main theorem of this paper.

\begin{theorem}\label{CASE 1}Let $G$ be an almost simple group and let $x,y$ be in $G\setminus\{1\}$. Suppose that the socle of $G$ is not isomorphic to an orthogonal group $\mathrm{P}\Omega_{n}^+(q)$, with $n\geqslant 8$. If $G=\nor G {\langle x\rangle}\nor G {\langle y\rangle}$, then (replacing $x$ by $y$ if necessary) $(G,x,y)$ is one of the triples in Table~$\ref{table1}$. See Sections~$\ref{notation}$~and~$\ref{notationtables}$, for the notation in Table~$\ref{table1}$.

Moreover, $G={\bf C}_G(x){\bf C}_G(y)$ if and only if in the $6^{\mathrm{th}}$ column of Table~$\ref{table1}$ appears the symbol $\surd$.
\end{theorem}

In the course of the proof of Theorem~\ref{CASE 1}, we show that every triple $(G,x,y)$ in Table~\ref{table1} gives rise to a genuine example of a factorization $G=\nor G {\langle x\rangle}\nor G{\langle y\rangle}$.

\begin{table}
\begin{adjustbox}{angle=90}
\begin{tabular}{|c|c|c|c|c|c|}\hline
Line&Group&element $x$&element $y$&Remarks&\\\hline
1&$\Sym(n)$& transposition&  $n$-cycle&$n$ prime&\\
2&$\Sym(5)$&$|x|\in \{3,6\}$& $5$-cycle &&\\\hline
3&$\PGL_2(r)$&$|x|=r$ & $y$ has no $1$-dim. eigenspace&&\\
4&$\PSL_2(r)$&$|x|=r$ &$y$ has no $1$-dim. eigenspace& $r\equiv 3\pmod 4$&\\
5&$\PGammaL_2(16)$& field aut. of order $2$& $|y|=17$&&\\

6&$\PSL_n(q)\unlhd G$& graph aut. of order $2$ & $|y|$ divides $q-1$, $y$ has an      & $n\geqslant 4$ even&$\surd$ \\

&                   & $\cent {\mathrm{PGL}_n(q)} x\cong \mathrm{PGSp}_n(q)$ & $(n-1)$-dim. eigenspace                          &$G$ contains a graph aut.    &\\

7&$\PSL_n(4)\unlhd G$&$|x|=5$, $x$ has  no eigenvalue&$|y|=3$, $y$ has an & $n\geqslant 4$ even, $G\nleq \mathrm{PGL}_n(4)\langle\tau\rangle$& \\

&&in $\mathbb{F}_q$&$(n-1)$-dim. eigenspace&$\tau$ inverse-transpose aut.&\\\hline

8&$\PSU_n(4)\unlhd G$&$|x|=3$, $x$ has no &$|y|=5$, $y$ has an $(n-1)$-dim.&$n$ even, &\\
&&eigenvalue in $\mathbb{F}_{q^2}$&eigenspace&$4$ divides $|G:L|$&\\

9&$\PSU_n(q)\unlhd G$& $|x|=2$, $x\in\mathrm{P}\Gamma\mathrm{U}_n(q)\setminus\PGU_n(q)$&  $|y|\mid q+1$, $x$ has an     & $n\geqslant 4$ even&$\surd$ \\
&                   &$\cent L x\cong \PSp_n(q)$   &    $(n-1)$-dim. eigenspace                      &$2$ divides $|G: L|$ &   \\\hline

10&$\mathrm{PSp}_n(q)\unlhd G$& $|x|=2$, $x\in \mathrm{PGSp}_n(q)\setminus \PSp_n(q)$
    & $|y|= r$, $y$ transvection  & $q$ odd, $n/2$ even &$\surd$\\
 
&                   &$\cent {L} x\cong \PSp_{n/2}(q^2).2$  &                             &$\mathrm{PGSp}_n(q)\leqslant G$&\\\hline
11&$\POmega_{n}^-(q)\unlhd G$& graph aut. of order 2& $|y|$ divides $q+1$, $y$ has no & $n/2$ odd&$\surd$\\
  &                        &$\cent Lx\cong \Omega_{n-1}(q)$  if $q$ odd                          & eigenvalue in $\mathbb{F}_{q}$                &$G$ contains a graph aut.&\\
&&$\cent Lx\cong \Sp_{n-2}(q)$ if $q$ even&&&\\
12&$\Aut(\Omega_{n}^-(4))$&$|x|=3$, $x$ has an   &       $|y|=5$, $y$ has no eigenvalue&$5$ divides $n$, $n/2$ odd &\\
  &                       &$(n-2)$-dim. eigenspace& in $\mathbb{F}_{q}$                       &          &\\
13&
$\mathrm{Aut}(\Omega_{n}^-(q))$
&$|x|=2$, $x\in\mathrm{SO}_n^-(q)\setminus\Omega_n^-(q)$&       $|y|=q^2+1$, $y$ has no eigenvalue&$n\equiv 4\pmod 8$ &\\
  &                       &${\bf C}_{\Omega_n^-(q)}(x)\cong\mathrm{Sp}_{n-2}(q)$& in $\mathbb{F}_{q^2}$, ${\bf C}_{\Omega_n^-(q)}(y)\cong\mathrm{GU}_{n/4}(q^2)$                       &   $q\in \{2,4\}$       & \\
  \hline
14&$\POmega_n(q)\unlhd G$&$|x|=2$, $\cent Lx\cong \POmega_{n-1}^-(q).2$ &          $|y|=r$, $y$ unipotent & $n\equiv 1\pmod 4$&$\surd$\\
  &                       &$x\in \SO_n(q)\setminus \Omega_n(q)$& $\cent L y\cong E_q^{m(m-1)/2+m}:\mathrm{Sp}_m(q)$                       &$\SO_n(q)\leqslant G$ &  \\\hline
\end{tabular}
\end{adjustbox}
\caption{Triples in Theorem~\ref{CASE 1}, $L\ne \POmega_n^+(q)$. See Sections~\ref{notation} and~\ref{notationtables} for notation.}\label{table1}
\end{table}

An immediate application of Theorem~\ref{CASE 1} gives the following corollary.

\begin{corollary}\label{cor2}Let $G$ be an almost simple transitive permutation group on $\Omega$ and let $\omega$ be in $\Omega$. Suppose that the socle of $G$ is not isomorphic to an orthogonal group $\mathrm{P}\Omega_{n}^+(q)$, with $n\geqslant 8$.
 If the point stabilizer $G_\omega$ normalizes a non-identity cyclic subgroup $\langle x\rangle$ and if $G$ contains an element $y\neq 1$ with $\nor G {\langle y\rangle}$ transitive on $\Omega$, then (replacing $x$ by $y$ if necessary) the triple $(G,x,y)$ is in Table~$\ref{table1}$.
\end{corollary}

A similar investigation for almost simple groups having socle an orthogonal group $\mathrm{P}\Omega_{n}^+(q)$, with $n\geqslant 8$, seems difficult and, at the moment, we do have 12 different families of factorizations using normalizers. We intend to come back to this question in the future.

It is worth mentioning that our strategy for proving Theorem~\ref{CASE 1}  is considerably different from the original proof of Szep's conjecture~\cite{FA}. Our main tool uses the classification of the maximal factorizations of the almost simple groups obtained by M.~Liebeck, C.~Praeger and J.~Saxl in~\cite{LPS1,LPS2}.

Let $G$ be an almost simple group with \textit{\textbf{socle}} $L$. A factorization $G=AB$ is said to be \emph{\textbf{maximal}} if $A$ and $B$ are both maximal subgroups of $G$, and is said to be \emph{\textbf{core-free}} if $A$ and $B$ are core-free in $G$ (that is, $L\nleq A,B$).
All the core-free maximal factorizations of the almost simple groups are classified in Tables~1--6 and Theorem~D of~\cite{LPS1}. In particular, if $G=\nor G {\langle x\rangle}\nor G {\langle y\rangle}$ (for some $x,y\in G\setminus\{1\}$) and $\nor G {\langle x\rangle}\leqslant A$, $\nor G {\langle y\rangle}\leqslant B$ for some core-free maximal subgroups $A$ and $B$ of $G$, then $(G,A,B)$ is one of the triples classified in~\cite{LPS1}. In particular, this reduces the proof of Theorem~\ref{CASE 1} to a case-by-case analysis on Tables~$1$--$6$ and on Theorem~D of~\cite{LPS1}. Moreover, for each of these triples $(G,A,B)$, we have $\nor G {\langle x\rangle }= \nor A {\langle x\rangle }$ and $\nor G {\langle y\rangle }=\nor B {\langle y\rangle }$ and so it suffices to investigate the order and the structure of the normalizers of the non-trivial elements of $A$ and $B$, respectively.

There is only one more case to consider in our analysis: every maximal subgroup of $G$ containing $\nor G {\langle x\rangle}$ (or $\nor G {\langle y\rangle}$) contains the socle $L$ of $G$. The almost simple groups admitting such factorizations are classified in~\cite[Table~$1$]{LPS2} and there is only a handful of such examples.

In the process of proving Theorem~\ref{CASE 1} using the work in~\cite{LPS1,LPS2}, we have realized  that there is one configuration omitted in the proof of Liebeck, Praeger and Saxl~\cite{LPS1,LPS2} classifying the maximal factorizations of the almost simple groups with socle $\mathrm{P}\Omega_8^+(2^f)$. Although this missing configuration is of no concern to us here because we are excluding almost simple groups having socle $\mathrm{P}\Omega_8^+(2^f)$ in our main results, we discuss this configuration in Section~\ref{sec:omission} and we show that this configuration does give rise to two extra maximal factorizations omitted in the work of Liebeck, Praeger and Saxl. In Section~\ref{sec:omission}, we comment how these extra factorizations influence other work relying on the classification in~\cite{LPS1,LPS2}.
The factorizations in \cite{LPS1} have been extensively used. For instance, recently,  this was used in \cite{LX} to give a characterization of the factorizations of almost simple groups with a solvable factor, which was then applied to study s-arc-transitive Cayley graphs of solvable groups, leading to a striking corollary that, except for cycles, a non-bipartite connected 3-arc-transitive Cayley graph of a finite solvable group is necessarily a normal cover of the Petersen graph or the Hoffman-Singleton graph. However, Zhou \cite{Zhou} improved this and obtained a remarkable refinement, that is, every non-bipartite connected Cayley graph of a finite solvable group is at most 2-arc-transitive.

\subsection{Notation}\label{notation}We use the notation from~\cite[Chapter~4]{GLS} and~\cite{BGbook} for conjugacy classes of elements in groups of Lie type and, in general, we use the notation from~\cite{KL} for the subgroups of the classical groups.

Given an almost simple group $G$, we denote by $L$ the socle of $G$. Suppose that $L$ a simple classical group defined over the finite field of size $q$.
For twisted groups our notation for $q$ is such that $\PSU_n(q)$ and $\POmega_{n}^-(q)$ are the twisted groups contained in $\PSL_n(q^2)$ and $\POmega_{n}^+(q^2)$, respectively.
We write $q=r^f$, for some prime $r$ and some $f\geqslant 1$, and we define
\begin{equation*}
q_0:= \begin{cases}
q^{2} & \mbox{if $G$ is unitary,}\\
q & \mbox{otherwise.}
\end{cases}
\end{equation*}
We let $V$ be the natural module defined over the field $\mathbb{F}_{q_0}$ of size $q_0$ for the covering group of $L$, and we let $n$ be the dimension of $V$ over $\mathbb{F}_{q_0}$.

We consider the
following \textbf{\em classical groups} $\tilde{L}$:
\begin{itemize}
\item$\mathrm{SL}_n(q)$ with $n\geqslant 1$,
\item$\mathrm{SU}_n(q)$  with $n\geqslant 1$,
\item$\mathrm{Sp}_{n}(q)$  with $n$ even and $n\geqslant 2$,
\item$\Omega_{n}(q)$  with $qn$ odd and $n\geqslant 1$, and
\item$\Omega_{n}^{\pm}(q)$ with $n$ even and $n\geqslant 2$.
\end{itemize}

For some of our proofs, we need to deal with arbitrary classical groups as defined above and hence with no restrictions on $n$. However, for proving our main results, we take into account the various isomorphisms among classical groups, see~\cite[Section~2.9]{KL}. For instance, $\mathrm{SL}_2(q)\cong\mathrm{SU}_2(q)\cong \mathrm{Sp}_2(q)\cong \Omega_3(q)$ and $\Omega_5(q)\cong \mathrm{Sp}_4(q)$. In particular, in Table~\ref{table1} and in proving Theorem~\ref{CASE 1}, we may suppose $n\geqslant 2$ for linear groups,  $n\geqslant 3$ for unitary groups, $n\geqslant 4$ for symplectic groups, $n\geqslant 7$ for odd dimensional orthogonal groups and $n\geqslant 8$ for even dimensional orthogonal groups.

The corresponding
\textbf{\em simple classical groups} $L:=\tilde L/ {\bf Z} (\tilde L)$ are
$$
\mathrm{PSL}_n(q),\,
\mathrm{PSU}_n(q),\,
\mathrm{PSp}_{n}(q),\,
\mathrm{P}\Omega_{n}(q),\hbox{ and }
\mathrm{P}\Omega_{n}^{\pm}(q).$$
With the restrictions on $n$ as above, these are indeed non-abelian simple groups, except for $\mathrm{PSL}_2(2)$, $\mathrm{PSL}_2(3)$, $\mathrm{PSU}_3(2)$ and $\mathrm{PSp}_4(2)$.

We denote by $\pi:\tilde L\to L$ the natural projection of $\tilde L$ onto $L$. By abuse of notation,  we refer to the action of $\tilde{L}$ on $V$ simply as the action $L$ on $V$. We adopt a similar convention for every $G$ with $L\unlhd G\leqslant \Aut(L)\cap \PGL(V)$. For example, for a subgroup $H$ of $L$, we say that $H$ acts irreducibly on $V$ when this is true of $\pi^{-1}(H)$.

Given an integer $\kappa$ and a prime number $p$, we write $\kappa_p$ for the \emph{\textbf{largest power}} of $p$ dividing $\kappa$. Given two integers $\kappa$ and $\kappa'$, we denote by $\gcd(\kappa,\kappa')$ the \emph{\textbf{greatest common divisor}} of $\kappa$ and $\kappa'$.

Given  a prime power $q$ and an integer $n\geqslant 2$, a prime $t$ is called a \textit{\textbf{primitive prime divisor}} of $q^n-1$ if $t$ divides $q^n-1$ and $t$  does not divide $q^i-1$, for each $i\in \{1,\ldots,t-1\}$.
From a celebrated theorem of Zsigmondy~\cite{Zs}, the following hold
\begin{itemize}
 \item for $n\geqslant 3$,  primitive prime divisors exist with the only exception of $(n,q)=(6,2)$,
\item for $n=2$, primitive prime divisors exist with the only exception of $q$ being a Mersenne prime, that is, $q$ is prime and $q=2^\ell-1$ for some $\ell\in\mathbb{N}$.
\end{itemize}
Note that, if $t$ is a primitive prime divisor of $q^n-1$, then $q$
has order $n$ modulo $t$ and thus $n$ divides $t-1.$ 

\subsection{Notation for Table~$\ref{table1}$.}\label{notationtables}
In reading Table~$\ref{table1}$, we take into account the notation we have established in Section~\ref{notation} and some isomorphisms among classical groups. 

When $L=\mathrm{PSL}_n(q)$, we suppose $n\geqslant 2$ and, when $$(n,q)\in \{(2,4),(2,5),(2,9),(4,2)\},$$ we refer to Lines~1 and~2 of Table~\ref{table1}, because $\mathrm{PSL}_2(4)\cong \mathrm{PSL}_2(5)\cong\mathrm{Alt}(5)$, $\mathrm{PSL}_2(9)\cong \mathrm{Alt}(6)$ and $\mathrm{PSL}_4(2)\cong\mathrm{Alt}(8)$. Moreover, when $(n,q)=(3,2)$, we refer to Lines~3 and~4, because $\mathrm{PSL}_3(2)\cong \mathrm{PSL}_2(7)$.

When $L=\mathrm{PSU}_n(q)$, we suppose $n\geqslant 3$;  when $L=\mathrm{PSp}_n(q)$, we suppose $n\geqslant 4$; when $L=\mathrm{P}\Omega_n(q)=\Omega_n(q)$ with $n$ odd, we suppose $n\geqslant 7$; when $L=\mathrm{P}\Omega_n^\pm(q)$ with $n$ even, we suppose $n\geqslant 8$.

\subsection{Structure of the paper}\label{subsection:intro}

In Section~\ref{sec:1}, we collect some basic results which we use throughout the whole paper, sometimes without mention.

In Section~\ref{sec:split}, we prove Theorem~\ref{CASE 1} for the almost simple groups having socle a sporadic, or an exceptional, or an alternating group. 

In the rest of the paper we deal with the classical groups. In Section~\ref{sec:classicalpsl}, we consider the linear groups $\mathrm{PSL}_n(q)$. In Section~\ref{sec:classicalpsu}, we consider the unitary groups $\mathrm{PSU}_n(q)$. In Section~\ref{sec:classicalpsp}, we consider the symplectic groups $\mathrm{PSp}_n(q)$. In Section~\ref{sec:classicalodd}, we consider the odd dimensional orthogonal groups $\mathrm{P}\Omega_n(q)=\Omega_n(q)$. In Section~\ref{sec:classicaloeven-}, we consider the even dimensional orthogonal groups $\mathrm{P}\Omega_n^-(q)$ having Witt defect $1$. 

\section{An additional maximal factorization of an almost simple group}\label{sec:omission}

 A  computation with the computer algebra system  \textsc{Magma}~\cite{magma} yields that there are  factorizations
 \begin{equation}\label{missing1}
 \Omega_8^+(4).\langle \phi \rangle=N_2^-\cdot \mathrm{SO}_8^-(2)=\mathrm{SO}_8^-(2)\cdot N_2^-,
 \end{equation}
 and
  \begin{equation}\label{missing2}
 \Omega_8^+(16).\langle \phi \rangle=N_2^-\cdot (\mathrm{SO}_8^-(4).2)=(\mathrm{SO}_8^-(4).2)\cdot N_2^-,
 \end{equation}
 where the subgroup $\SO_8^-(q^{1/2})\leqslant\Omega_8^+(q)$ is the image of a $\mathcal{C}_5$ subgroup under a triality automorphism, and $\phi$ is a non-identity field automorphism of order $f$, where $q=2^f$. Recall that we are using the notation in~\cite{KL} and hence, in particular, $N_2^-$ is the stabilizer of a $2$-dimensional anisotropic subspace of $V=\mathbb{F}_q^8$. 
 Moreover, these factorizations do not lead to a factorization of the simple group $L=\Omega_8^+(q)$; actually, these factorization exists only in the almost simple groups $\mathrm{Aut}(\Omega_8^+(q))$-conjugate to $\Omega_8^+(q).\langle\phi\rangle$ and in no other almost simple groups with socle $\Omega_8^+(q)$. Observe that, using triality, we have exactly three possibilities for $\Omega_8^+(q).\langle\phi\rangle$.  Both factors in each of these factorisations are core-free maximal subgroups and so these factorisations are max+ factorisations in the terminology of \cite{LPS2}. Thus \cite[Table 4]{LPS1} should have the rows of Table \ref{table:extra} added.

\begin{table}[!ht]
\begin{tabular}{|llllp{4cm}l|}\hline
$L$ & $*$ or $\dagger$& $A\cap L$ & $B\cap L$ & Remark & $Y$ column\\\hline
$\Omega_8^+(4)$ & $*$ &$(5\times \Omega_6^-(4)).2$ & $\Omega_8^-(2)$ & $G= L.2$, $A$ in $\mathcal{C}_1$ or $\mathcal{C}_3$, $B$ in $\mathcal{C}_5$ or $\mathcal{C}_9$ depending on choice of $A$. Two possible $B$ for each $A$. Moreover, $G$ contains a field automorphism.  &\\ &&&&& \\
$\Omega_8^+(16)$ & $*$ &$(17\times \Omega_6^-(16)).2$ & $\Omega_8^-(4)$ & $G= L.4$, $A$ in $\mathcal{C}_1$ or $\mathcal{C}_3$, $B$ in $\mathcal{C}_5$ or $\mathcal{C}_9$ depending on choice of $A$. Two possible $B$ for each $A$. Moreover, $G$ contains a field automorphism of order 4.  &\\\hline

\end{tabular}
\caption{Missing maximal factorizations}\label{table:extra}
\end{table}

For $q=4$, the factorisation can be verified by the following steps:
 \begin{enumerate}
     \item construct the action of $\Omega^+_8(4)$ on 2-dimensional subspaces of ``minus type'';
     \item find the normaliser of the induced permutation group in $\Sym(6580224)$ to obtain the permutation group $G$ for  $\Omega^+_8(4).\langle \phi \rangle$;
     \item use the ClassicalMaximals command to construct an appropriate $H=\Omega_8^-(2)$ in $\Omega_8^+(4)$;
     \item find the image of $H$ in $G$ and then determine its normaliser in $G$ to find the appropriate $\SO_8^-(2)$ subgroup;
     \item check that the $\SO_8^-(2)$ is transitive in this action.
 \end{enumerate}
 For $q=16$ the groups are too large to do  many of these steps. It is possible though to use the ClassicalMaximals command to construct an appropriate $\Omega_8^-(4)$ in $\Omega^+_8(16)$ and then show that it  has an orbit on the set of 2-dimensional subspaces of ``minus type'' whose length is one quarter of the total number of such subspaces. We then  explicitly construct $\mathrm{SO}_8^-(4).2$ in  $\Omega_8^+(16).\langle \phi \rangle$, and the Sylow 2-subgroups of the two potential factors. We can then exhibit that the intersection of the Sylow 2-subgroups has order 8 and is contained in $\Omega_8^-(4)$ and so we do indeed have a factorisation.


\subsection{Dealing with the missing factorization}Using the notation in~\cite{LPS1}, the factorizations in~\eqref{missing1} and \eqref{missing2} were erroneously ruled out in~\cite[p106--107]{LPS1} when considering the possibility $$A\cap L=((q+1)/d\times\Omega_6^-(q)).2^d/Z \hbox{ and }B\cap L=\Omega_8^-(q^{1/2}),$$ where $d:=\gcd(2,q-1)$ and $Z$ is the group of scalars in $\Omega_8^+(q)$. The argument there for $q$ even only rules out a factorization of the simple group $L$ and misses a subtlety due to triality. We now provide a complete analysis of this case following that in~\cite{LPS1}.

Let $G$ be an almost simple group having socle $L:=\mathrm{P}\Omega_8^+(q)$ and let $A$ and $B$ be maximal subgroups of $G$. Suppose that $G=AB$ with $A\cap L=((q+1)/d\times\Omega_6^-(q)).2^d/Z$ and $B\cap L=\Omega_8^-(q^{1/2})$. As in Section~\ref{notation}, we let $V$ be an 8-dimensional vector space over the finite field $\mathbb{F}_q$ equipped with a non-degenerate quadratic form of plus type with $q=r^f$ for some prime $r$ and some even positive integer $f$. By applying a suitable triality automorphism if necessary we may assume that $A=N_2^-$, the stabilizer in $G$ of an anisotropic 2-dimensional subspace. Moreover, by \cite{Kleidman2} the maximality of $A$  in $G$ implies that $G\leqslant  \mathrm{P}\GammaO_8^+(q)$. Let $X\cong\Omega_8^-(q^{1/2})$ be a subgroup of $L$ in $\mathcal{C}_5$, that is, a subfield subgroup. By \cite{Kleidman2}, we may take $B\cap L=X^{\tau^a}$ for some $a\in\{0,1,2\}$, where $\tau$ is a triality automorphism of $L$. Write $k=\tau^a$. Now $X$ has a subgroup $$Y=(\mathrm{SO}_4^+(q^{1/2})\times\mathrm{SO}_4^-(q^{1/2}))\cap X$$ fixing an orthogonal sum decomposition $V=W_1\perp W_2$, where $W_1$ and $W_2$ are both non-degenerate 4-dimensional subspaces of $V$ of plus type. By \cite[Lemma~4.1.1]{KL}, we have that
$$Y=(\Omega_4^+(q^{1/2})\times\Omega_4^-(q^{1/2})).2$$
and $Y$ induces $\SO_4^+(q^{1/2})$ and $\SO_4^-(q^{1/2})$ on $W_1$ and $W_2$ respectively. Therefore, the kernel of the action of $Y$ on $W_1$ is $\Omega_4^-(q^{1/2})$ and the kernel of the action of $Y$ on $W_2$ is $\Omega_4^+(q^{1/2})$.
Let $M$ be the stabilizer in $L$ of this decomposition. Now $Y^k\leqslant  B\cap L$ and $Y^k\leqslant  M^k$. By \cite{Kleidman2}, $\tau$ fixes setwise the $L$-conjugacy class of $M$ in $L$ and so $W^k$ fixes an orthogonal  decomposition $V=W_1'\perp W_2'$, with $W_1',W_2'$ both 4-dimensional non-degenerate subspaces of $V$ of plus type. Note that $Y$ has a subgroup of index 2 and so potentially $Y^k$ interchanges $W_1'$ and $W_2'$. Indeed a \textsc{Magma} \cite{magma} calculation shows that this does indeed happen when $q=4$. 
\begin{center}
This seems to have been overlooked by \cite{LPS1} as their argument seems to assume that $Y^k$ fixes both $W_1'$ and $W_2'$.
\end{center} 
We now continue the analysis of this potential factorization, obtaining the missing factorizations in~\cite{LPS1}.

For $i=1,2$, let $K_i$ be the kernel of the action of the stabilizer in  $Y^k$ of $W_i'$ on $W_i'$.  Note that  $(Y^k)_{W_i'}$ induces a subgroup of $\GO^+_4(q)$ on $W_i'$. Now $\Omega_4^+(q^{1/2})\cong\SL_2(q^{1/2})\circ\SL_2(q^{1/2})$ and $\Omega_4^-(q^{1/2})=\PSL_2(q)$.  Since $\GO_4^+(q)$ does not contain a subgroup isomorphic to an index two subgroup of $Y$, each $K_i$ is nontrivial. 

Suppose first that $q=r^f$ is odd.  For $q\neq 9$, by considering the normal subgroups of $Y'$, the unique index 2 subgroup of $Y$, we see that either $\SL_2(q^{1/2})$ or $\Omega_4^-(q^{1/2})$ lies in $K_1$.  On the other hand, for $q=9$, we see that the centralizer in $\GO^+_4(9)$ of $\Omega^-_4(3)$ does not contain an element of order 3 and so we may draw the same conclusion about $K_1$. Then, taking a suitable $2$-dimensional subspace $U\leqslant  W_1'$ and $A:=G_U=\{g\in G\mid U^g=U\}$, we have that $K_1\leqslant  A\cap B$ and so $q^{1/2}$ divides $|A\cap B|$. As $q$ is odd, we have that 
\begin{align*}
|A|_r&= |\Omega_6^-(q)|_r|G:L|_r=|G:L|_rq^6,\\
|B|_r&=|\Omega_8^-(q^{1/2})|_r|G:L|_r=q^6|G:L|_r,\\
|A\cap B|_r&\geqslant q^{1/2},\\
|G|_r&=q^{12}|G:L|_r.
\end{align*} As $G$ does not contain a triality automorphism of $L$, we have that $|G:L|_r$ divides $f_r$ and hence we see that $|A|_r|B|_r < |G|_r|A\cap B|_r$, contradicting $G=AB$. Therefore, when $q$ is odd, there are no factorizations, as predicted by~\cite{LPS1}.

Now suppose that $q$ is even. Then $|A|_2= 2|\Omega_6^-(q)|_2|G:L|_2=2q^6|G:L|_2$ and $|B|_2= q^6|G:L|_2$, while $|G|_2=q^{12}|G:L|_2$. Since $G=AB$, this implies that 
$|A\cap B|_2=  2|G:L|_2.$ Note also that 
\begin{equation}\label{eq:condition}
|G:L|_2\hbox{ divides }2f_2.
\end{equation}  Suppose first that $B\cap L=X$.  Then $W_1'=W_1$ and $W_2'=W_2$. Choose $U\leqslant  W_1$ to be a 2-dimensional subspace of minus type and take $A:=G_U$. Since $Y^{W_1}=\mathrm{SO}^+_4(q^{1/2})$ and the order of the stabilizer in $\SO_4^+(q^{1/2})$ of a non-degenerate 2-dimensional subspace is divisible by 4, 
 we see that $|A\cap B|_2\geqslant |Y_U|_2=4q$. However, this contradicts~\eqref{eq:condition}, because $|A\cap B|_2= 2|G:L|_2\leqslant  4f_2$. Thus $$B\cap L=X^{\tau^a},$$ where $a\in\{1,2\}$.
When $q=4$ and 16, we have verified with \textsc{Magma} that we do obtain the maximal factorizations in Table~\ref{table:extra}.  Suppose then $q>16$. Thus $q\geqslant 64$, because $f$ is even.

 Then, by~\cite{Kleidman2}, the maximality of $B$ in $G$ implies that $|G:L|$ divides $f$. As $q\geqslant 64$, the group $Y$ has a unique index 2 subgroup, namely $Y'$. Then the kernel of the action of $Y'$ on $W_1'$ is non-trivial and it is not hard to show that $K_1\geqslant \SL_2(q^{1/2})$. Therefore, $$((Y')^{\tau^a})^{W_1}=\mathrm{SL}_2(q^{1/2})\times\mathrm{SL}_2(q).$$ We claim that the stabilizer in this group of a 2-dimensional subspace of $W_1$ of minus type has even order.  We argue by contradiction, and we suppose that there exists a $2$-dimensional subspace $U$ of minus type of $W_1$ with the property that the stabilizer in $\SL_2(q^{1/2})\times\SL_2(q)$ has odd order. Let us denote by $S_1$ and $S_2$ the two simple direct factors of $\Omega_4^+(q)$. Let $N:=\SL_2(q^{1/2})\times\SL_2(q)$ and let $\Omega$ be the collection of all $2$-dimensional subspaces of $W_1$ of minus type. Routine computations yield
 \begin{equation}\label{fail}|\Omega|=\frac{q^2(q-1)^2}{2}.
 \end{equation}  Without loss of generality we may suppose that $S_2\subseteq M$. Observe that 
 $$|(\Omega_2^-(q)\times\Omega_2^-(q)).2|=2(q+1)^2 \hbox{ and }|N|=q^{3/2}(q-1)(q^2-1).$$
Since $\gcd(q+1,q-1)=1$ and since we are assuming that $M_U$ has odd order, we deduce that $|N_U|$ divides $q+1$ and $N_U\leqslant S_2$. From one hand we deduce that the $N$-orbit containing $U$ has cardinality divisible by $q^{3/2}$ and, on the other hand, we deduce that $S_1$ centralizes $N_U$.
 As $\Omega_4^+(q)=S_1N$, we deduce that each orbit of $M$ on the $2$-dimensional subspaces of $W_1$ of minus type has order divisible by $q^{3/2}$. The number of $N$-orbits is thus $|S_1:S_1\cap N|$ and hence $q^{1/2}\cdot q^{3/2}=q^2$ divides $|\Omega|$; however, this contradicts~\eqref{fail}.

\subsection{Impact of the missing factorizations in other work}
The factorizations in \cite{LPS1} have been extensively used. Now, we comment how this extra factorization influences other work relying
in the classification in~\cite{LPS1}.
\begin{enumerate}
\item Since the $Y$ column for the new factorizations are empty, these new factorizations do not give an example of a primitive almost simple group of degree $n$ as a proper subgroup of a primitive almost simple group of the same degree but different socle. Thus no new examples arise for  \cite[Table VI]{LPSmaxsubs}.
\item The classification of maximal factorizations of almost simple groups was used in \cite{LPSregsubs} to determine all regular subgroups of the primitive almost simple groups. These new factorizations provide four new primitive almost simple groups with a core-free transitive subgroup (namely the action of $G$ on the set of cosets of $A$ and the action of $G$ on the set of cosets of $B$). We have checked with the help of a computer that none of these  new primitive actions admits regular subgroups and hence no exception arises in~\cite{LPSregsubs}.

\item For these new factorizations $|G:A|$ and $|G:B|$ are even and so these are not coprime factorizations and so no new factorization needs to be added to \cite[Table 1]{nostro}.

\item  The maximal factorisations in \cite{LPS1} are used in \cite{LX} to determine the factorisations of almost simple group with one of the two factors solvable. Therefore, in principle, the missing factorizations arising when the socle is $\mathrm{P}\Omega_8^+(4)$ and $\mathrm{P}\Omega_8^+(16)$ could in principle yield factorizations missed by \cite{LX} when using \cite{LPS1}. We have checked with \textsc{Magma} and we confirm that no new factorization arises when one of the two factors is solvable.

\end{enumerate}

\section{Preliminary remarks}\label{sec:1}

\begin{lemma}\label{l: smaller}
Let $G$ be a group, let $M$ and $N$ be subgroups of $G$ with $G=MN$ and let $T$ be a subgroup of $N$.
Then $G=MT$ if and only if $N=(M\cap N)T.$
\end{lemma}
\begin{proof}
If $G=MT$,  then $N=MT\cap N=(M\cap N)T$, as required.
Conversely, if $N=(M\cap N)T$, then $G=MN =M(M\cap N)T = MT.$
\end{proof}

The following elementary lemma is one of the ingredients for  our proof of Theorem~\ref{CASE 1} when the socle of $G$ is an alternating group. A permutation $g\in \Sym(n)$ is said to be {\em \textbf{semiregular}} if all  the orbits of $\langle g\rangle$ on $\{1,\ldots,n\}$ have the same length.

\begin{lemma}\label{le3}Let $g$ be in $\Sym(n)\setminus\{1\}$. If $\nor {\Sym(n)}{\langle g\rangle}$ is transitive on $\{1,\ldots,n\}$, then $g$ is semiregular.
\end{lemma}
\begin{proof}
Write $N:=\nor {\Sym(n)}{\langle g\rangle}$. Clearly, $N$ permutes the orbits of $\langle g\rangle$ having the same length. Since $N$ is transitive on $\{1,\ldots,n\}$, we obtain that all $\langle g\rangle$-orbits have the same length.
\end{proof}

For the rest of this paper, $G$ denotes an almost simple group with socle $L$ having two group elements $x,y\in G\setminus\{1\}$ with $$G=\nor G {\langle x\rangle}\nor G {\langle y\rangle}.$$ Write 
$$X:=\nor G{\langle x\rangle}\hbox{ and }Y:=\nor G {\langle y\rangle},$$ for short.


\begin{lemma}\label{le1}We have $$G = \nor G { \langle x\rangle}\nor G { \langle y\rangle^g } = \nor G { \langle x\rangle^g }\nor G {\langle y \rangle},$$ for every $g \in G$.
Let  $\langle g_1\rangle,\ldots,\langle g_\ell\rangle $ be a set of representatives for the conjugacy classes of non-identity  cyclic subgroups of $G$. Then there exists $i,j\in \{1,\ldots,\ell\}$ with $G=\nor G {\langle g_i\rangle}\nor G {\langle g_j\rangle}$.
\end{lemma}
\begin{proof}
Let $g$ be in $G$. We have $g=uv$, for some $u\in X={\bf N}_G(\langle x\rangle)$ and $v\in Y={\bf N}_G(\langle y\rangle)$. Now,
\begin{align*}
\nor G{\langle x\rangle^g}\nor G {\langle y\rangle}&=\nor G {\langle x\rangle^v}\nor G {\langle y\rangle}=\nor G {\langle x\rangle^v}\nor G {\langle y\rangle^v}\\
&=\nor G {\langle x\rangle}^v\nor G {\langle y\rangle}^v=(XY)^v=G^v=G.
\end{align*}
The other case is similar. Now, the rest of the proof follows from the fact that $\langle x\rangle =\langle g_i\rangle^{h_x}$ and $\langle y\rangle=\langle g_j\rangle^{h_y}$, for some $i,j\in \{1,\ldots,\ell\}$ and $h_x,h_y\in G$.
\end{proof}

Lemma~\ref{le1} gives a very efficient test to check whether $G=\nor G { \langle x\rangle}\nor G { \langle y\rangle}$. For example, for $M_{11}$ it is immediate to see with~\cite{ATLAS} that $|\nor {M_{11}} {\langle g \rangle}|\leqslant 55$, for every $g\in M_{11}\setminus\{1\}$. As $|M_{11}|>55^2$, we see that Theorem~\ref{CASE 1} holds true for $M_{11}$, that is, $M_{11}$ has no factorization of the form $M_{11}={\bf N}_{M_{11}}(\langle x\rangle){\bf N}_{M_{11}}(\langle y\rangle)$ with $x,y\in M_{11}\setminus\{1\}$.

We also have the following useful lemma.

\begin{lemma}\label{lem:primeorder}
Suppose that $G={\bf N}_G(\langle x\rangle){\bf N}_G(\langle y\rangle)$ for some $x,y\in G$. Then 
$G={\bf N}_G(\langle x'\rangle){\bf N}_G(\langle y'\rangle)$ for some $x',y'\in G$ of prime order.
\end{lemma}
\begin{proof}
Let $p_x$ divide $|x|$ and $p_y$ divide $|y|$ be primes. Then $x':=x^{|x|/p_x}$ and $y':=y^{|y|/p_y}$ have prime order. Moreover, ${\bf N}_G(\langle x\rangle) \leqslant {\bf N}_G(\langle x'\rangle)$ and 
$ {\bf N}_G(\langle y\rangle)\leqslant {\bf N}_G(\langle y'\rangle)$. Thus the result follows.
\end{proof}

\section{Proof of Theorem~\ref{CASE 1} for sporadic, exceptional  and alternating simple groups}\label{sec:split}

We are now ready to prove Theorem~\ref{CASE 1} when $L$ is a sporadic simple group, an exceptional group of Lie type or an alternating group.

\begin{proposition}\label{pr1}$L$ is not a sporadic simple group.
\end{proposition}
\begin{proof}
All factorizations of almost simple groups having socle a sporadic simple group are determined in~\cite{mgfact}. The result follows by inspection.
\end{proof}

\begin{proposition}\label{pr2}$L$ is not an exceptional group of Lie type.
\end{proposition}
\begin{proof}
Suppose that $G=XY$ and by Lemma \ref{lem:primeorder} we may assume that both $x$ and $y$ have prime order. From~\cite[Table~$5$]{LPS1}, we see that $L$ is one of the following groups:
\begin{itemize}
\item $G_2(q)$ with $q=3^f$,
\item $F_4(q)$ with $q=2^f$,
\item $G_2(4).$
\end{itemize}
Note that \cite[Table~$5$]{LPS1} gives all factorisations of $G$, not just the maximal ones, and so lists the possibilities for $X$ and $Y$. 

Suppose first that $L=G_2(4)$. Then interchanging $X$ and $Y$ if necessary we have that $X\cap L=\SU_3(4).4\cap L$. Since $X\cap L$ has trivial centre it follows that $x\notin L$ and so is a field automorphism of order 2. However, this implies that $C_L(x)\cong G_2(2)$, which does not contain $\SU_3(4)$, a contradiction.

Next suppose that $L=G_q(q)$ with $q=3^f$. Then interchanging $X$ and $Y$ if necessary we have that $X\cap L=\SL_3(q)$ or $\SL_3(q).2$. In both cases $X\cap L$ has trivial centre and so $\SL_3(q)\leqslant {\bf C}_A(x)$. Moreover, $x\notin L$. Hence by \cite[Proposition 4.9.1]{GLS} $x$ is either a field or graph-field automorphism of $L$ and hence $C_L(x)=G_2(q_0)$ with $q=q_0^e$, or ${}^2G_2(q)$, respectively. Neither of these contain $\SL_3(q)$ as a subgroup, a contradiction.
 
 Similarly, if $L=F_4(q)$ then $X\cap L=\Sp_8(q)$. Hence $\Sp_8(q)\leqslant {\bf C}_A(x)$ and $x\notin L$. Again,  $x$ is either a field or graph-field automorphism of $L$ and hence $C_L(x)=F_4(q_0)$ with $q=q_0^e$, or ${}^2F_4(q)$, respectively. Neither of these contain $\Sp_8(q)$ as a subgroup, and so we obtain another contradiction.
 
 \end{proof}

\begin{proposition}\label{pr3}
If $L=\Alt(n)$, then the triple $(G,x,y)$ is in Line~$1$ or~$2$ of Table~$\ref{table1}$.
\end{proposition}
\begin{proof}For $5\leqslant n\leqslant 6$, the result follows by a computation using Lemma~\ref{le1}. We obtain the examples in Lines~1 or~2 of Table~\ref{table1}. Assume that $n\geqslant 7$. In particular, $G=\Alt(n)$ or $G=\Sym(n)$. Then by \cite[Theorem D]{LPS1}, interchanging $X$ and $Y$ if necessary we have that either $\Alt(n-k)\unlhd X\leqslant \Sym(k)\times \Sym(n-k)$ for some $k\leqslant 5$ and $Y$ is $k$-homogeneous on $n$ points, or one of the following holds:
\begin{enumerate}
    \item $n=8$ and $X=\mathrm{AGL}_3(2)$;
    \item $n=10$ and $X=\PSL_2(8)$ or $\PSL_2(8).3$.
\end{enumerate}
In these two exceptional cases, $X$ is clearly not the normaliser of a nontrivial cyclic subgroup. Hence $\Alt(n-k)\unlhd X\leqslant \Sym(k)\times \Sym(n-k)$ and $Y$ is $k$-homogeneous. Note that since $X$ has a nontrivial cyclic normal subgroup we must have that $k\geqslant 2$. Hence by \cite[Theorem 6]{BP}, $Y$ is primitive on $n$ points. As $\langle y\rangle \unlhd Y$, we get that the socle of the primitive group $Y$ is $\langle y\rangle$. Thus $n$ is prime and $y$ is a cycle of prime order.  Moreover, $Y\cong\mathrm{AGL}(1,n)\cap G$ is 2-homogeneous but not 3-homogeneous. Hence either $k=2$, or $n=7$ and $k=5$. In the first case we deduce that $x$ is a transposition, $G=\Sym(n)$ and $X=\Sym(2)\times \Sym(n-2)$. Thus we do obtain a factorisation. It remains to consider the case where $k=5$ and $n=7$, and we may assume that $x$ is not a transposition. Since $|G|=|X||Y|/|X\cap Y|$ we deduce that $X$ contains a 5-cycle. As $x$ is not a transposition we deduce that $x$ is a 5-cycle. However, we then have $|X|\leqslant 40$ and $|Y|\leqslant 42$, which contradicts $|X||Y|\geqslant |G|\geqslant (7!)/2$.

It is easy to verify that in Line~1 and~2 we have ${\bf C}_G(x){\bf C}_G(y)<G$ and hence there is no symbol $\surd$ in the $6^{\mathrm{th}}$ column.
\end{proof}

\section{Classical groups: linear groups}\label{sec:classicalpsl}
In this section we assume that $G$ is an almost simple group with socle $L=\mathrm{PSL}_n(q)$ with $q=r^f$ for some prime $r$.

We start our analysis with two technical lemmas which help to locate the elements $x$ and $y$ with $G={\bf N}_G(\langle x\rangle){\bf N}_G(\langle y\rangle)$. Our main reference for these lemmas is~\cite[Chapter~4]{GLS} and~\cite[Chapter~$3.1$, Table~B.1,~B.2,~B.3]{BGbook}.
\begin{lemma}\label{hup0}Let $n\geqslant 2$. Suppose $r^{fn}-1$ admits a primitive prime divisor  $t_1$. Let $g\in \mathrm{Aut}(\mathrm{PSL}_n(q))$ with $t_1$ dividing $|{\bf N}_{\mathrm{Aut}(\mathrm{PSL}_n(q))}{(\langle g\rangle)}|$ and let $T_1$ be a cyclic subgroup of order $t_1$ in 
${\bf N}_{\mathrm{Aut}(\mathrm{PSL}_n(q))}{(\langle g\rangle)}$. Then 
\[
g\in {\bf C}_{\mathrm{Aut}(\mathrm{PSL}_n(q))}(T_1)=
\begin{cases}
\langle T,\iota\rangle&\iota \textrm{ graph automorphism}, \textrm{ if  }n \textrm{ is even},\\
\langle T, \iota\rangle&\iota \textrm{ graph-field automorphism}, \textrm{ if }n \textrm{ is odd and } f\textrm{ is even},\\
T&\textrm{if }nf\textrm{ is odd},
\end{cases}
\]
where $T$ is a maximal torus of $\mathrm{PGL}_n(q)$ having order $(q^n-1)/(q-1)$, that is, $T$ is a Singer cycle. In particular, $|{\bf N}_{\mathrm{Aut}(\mathrm{PSL}_n(q))}(\langle g\rangle):{\bf C}_{\mathrm{Aut}(\mathrm{PSL}_n(q))}(g)|$ is relatively prime to $t_1$.
\end{lemma}
\begin{proof}
Suppose first that $T_1\leqslant {\bf C}_{\mathrm{Aut}(\mathrm{PSL}_n(q))}{( g)}$. Then $g\in {\bf C}_{\mathrm{Aut}(\mathrm{PSL}_n(q))}{(T_1)}$. Let $T$ be a Singer cycle of $\mathrm{PGL}_n(q)$ containing $T_1$. Using~\cite{GLS}, we obtain the structure of ${\bf C}_{\mathrm{Aut}(\mathrm{PSL}_n(q))}(T_1)$.

It remains to consider the case that $t_1$ does not divide the order of ${\bf C}_{\mathrm{Aut}(\mathrm{PSL}_n(q))}{( g)}$. We aim to prove that this case cannot arise. As $t_1$ divides $|{\bf N}_{\mathrm{Aut}(\mathrm{PSL}_n(q))}{(\langle g\rangle)}|$ and as ${\bf N}_{\mathrm{Aut}(\mathrm{PSL}_n(q))}{(\langle g\rangle)}/{\bf C}_{\mathrm{Aut}(\mathrm{PSL}_n(q))}{(g)}$ acts faithfully as a group of automorphisms on the cyclic group $\langle g\rangle$, we deduce that $t_1$ divides $\varphi(|g|)$, where $\varphi$ is the Euler totient function. In particular, $|g|$ is divisible by a prime $p$ with $t_1\mid p-1$. Without loss of generality, replacing $g$ by $g^{|g|/p}$ if necessary, we may suppose that $|g|=p$. 
As $fn$ divides $t_1-1$, we have $fn<t_1$. As $t_1\mid p-1$ and $fn<t_1$, we deduce $g\in\mathrm{PSL}_n(q)$. Since $t_1$ is a primitive prime divisor for $r^{fn}-1$, we have $p\ne r$ and hence $g$ is semisimple. Now, as $T_1$ acts non-trivially on $\langle g\rangle$, we deduce that $T_1$ permutes non-trivially the eigenspaces of $g$. However, this is a contradiction because $t_1>n$.
\end{proof}

\begin{lemma}\label{hup}Let $n\geqslant 3$. Suppose $r^{f(n-1)}-1$ admits a primitive prime divisor $t_2$. Let $g\in \mathrm{Aut}(\mathrm{PSL}_n(q))$ with $t_2$ dividing $|{\bf N}_{\mathrm{Aut}(\mathrm{PSL}_n(q))}{(\langle g\rangle)}|$ and let $T_2$ be a cyclic subgroup of order $t_2$ in 
${\bf N}_{\mathrm{Aut}(\mathrm{PSL}_n(q))}{(\langle g\rangle)}$. Then one of the following holds:
\begin{enumerate}
\item\label{hup1}
\[
g\in {\bf C}_{\mathrm{Aut}(\mathrm{PSL}_n(q))}(T_2)=
\begin{cases}
\langle T,\iota\rangle&\iota \textrm{ graph automorphism}, \textrm{ if  }n \textrm{ is odd},\\
\langle T, \iota\rangle&\iota \textrm{ graph-field automorphism}, \textrm{ if }n \textrm{ and } f\textrm{ are even},\\
T&\textrm{if }(n-1)f \textrm{ is odd},
\end{cases}
\]
where $T$ is a maximal torus of $\mathrm{PGL}_n(q)$ having order $q^{n-1}-1$. In particular, $|{\bf N}_{\mathrm{Aut}(\mathrm{PSL}_n(q))}(\langle g\rangle):{\bf C}_{\mathrm{Aut}(\mathrm{PSL}_n(q))}(g)|$ is relatively prime to $t_2$;
\item\label{hup2}$n=t_2$ is prime, $f=1$, $g$ lies in a Singer cycle of $\mathrm{PGL}_n(q)$ and has order divisible by a primitive prime divisor $p$ of $r^{fn}-1$ with $n\mid p-1$.
\end{enumerate}
\end{lemma}
\begin{proof}
If $T_2\leqslant {\bf C}_{\mathrm{Aut}(\mathrm{PSL}_n(q))}{( g)}$, then the proof follows verbatim the argument in Lemma~\ref{hup0} and we obtain that part~\eqref{hup1} holds.

It remains to consider the case that $t_2$ does not divide the order of ${\bf C}_{\mathrm{Aut}(\mathrm{PSL}_n(q))}{( g)}$. We aim to prove that part~\eqref{hup2} holds. As $t_2$ divides $|{\bf N}_{\mathrm{Aut}(\mathrm{PSL}_n(q))}{(\langle g\rangle)}|$ and as ${\bf N}_{\mathrm{Aut}(\mathrm{PSL}_n(q))}{(\langle g\rangle)}/{\bf C}_{\mathrm{Aut}(\mathrm{PSL}_n(q))}{(g)}$ acts faithfully as a group of automorphisms on the cyclic group $\langle g\rangle$, we deduce that $t_2$ divides $\varphi(|g|)$. In particular, $|g|$ is divisible by a prime $p$ with $t_2\mid p-1$. Without loss of generality, we may suppose that $|g|=p$. 
As $f(n-1)$ divides $t_2-1$, we have $f(n-1)<t_2$ and hence $g\in\mathrm{PSL}_n(q)$. Now, as $T_2$ acts non-trivially on $\langle g\rangle$, we deduce that $T_2$ permutes non-trivially the eigenspaces of $g$. As $t_2>n-1$, this is possible only when $\langle g\rangle$ has $n$ distinct eigenvalues and $t_2=n$. As $t_2=n$ and $f(n-1)<t_2$, we have $f=1$. Moreover, as $g$ has $n$ distinct eigenvalues in a suitable extension of $\mathbb{F}_q$, we deduce that $g$ is contained in a Singer cycle of $\mathrm{PSL}_n(q)$ and, via the embedding of $\mathbb{F}_{q^n}^\ast$ in $\mathbb{F}_q^n\setminus\{0\}$, $g$ is a field generator.
\end{proof}

 In our proofs, we exclude those groups that are isomorphic to alternating groups, as these have already been covered in Proposition~\ref{pr3}. Thus $n\geqslant 2$ and $$(n,q)\not\in\{ (2,4), (2,5), (2,9),(4,2)\}.$$ See Section~\ref{notation}, for our notation.
 
We first deal with $2$-dimensional linear groups, because we have little room in this case for using the primitive prime divisors $t_1$ and $t_2$ in Lemmas~\ref{hup0} and~\ref{hup}.

\begin{lemma}\label{casePSL2}
 If $L=\PSL_2(q)$ and $q\notin\{4,5,9\}$, then $(G,x,y)$ is in Line~$3$,~$4$ or~$5$ of Table~$\ref{table1}$.
\end{lemma}
\begin{proof}
  Assume that $f=1$. Here $G=\PSL_2(r)$ or $G=\PGL_2(r)$. Now, replacing $X$ by $Y$ if necessary, we may assume that $r$ divides $|X|$. The only elements $x$ of $G$ having normalizer of order divisible by $r$ are the $r$-elements. Thus $|x|=r$, $X$ is a Borel subgroup of $G$ and 
 \[|X|=
 \begin{cases}
 \frac{(r-1)r}{2}&\textrm{if }G=\PSL_2(r),\\
 (r-1)r&\textrm{if } G=\PGL_2(r).
 \end{cases}
 \] As $G=XY$ and $X$ fixes a $1$-dimensional subspace of $V$, we see from the Frattini argument that $Y$ acts transitively on the $1$-dimensional subspaces of $V$. With a quick look at the structure of the conjugacy classes of $G$, we obtain that $y\in T\setminus\{1\}$, with $T$ a maximal torus of $\PGL_2(r)$ of order $r+1$.  In particular,  
 \[|Y|=|\nor G {\langle y\rangle}|=
 \begin{cases}
 r+1&\textrm{if }G=\PSL_2(r),\\
 2(r+1)&\textrm{if }G=\PGL_2(r).
 \end{cases}
\] If $G=\PGL_2(r)$, we have $$\frac{|X||Y|}{|X\cap Y|}=\frac{((r-1)r)(2(r+1))}{2}=|G|$$ and
hence $G=XY$: these examples are in Line~3 of Table~\ref{table1}.  If $G=\PSL_2(r)$, then $Y$ is a dihedral group with $|X\cap Y|=1$ if $r\equiv 3\pmod 4$ and with $|X\cap Y|=2$ if $r\equiv 1\mod 4$. In particular, $G=XY$ only when $r\equiv 3\pmod 4$: these examples are in Line~4 of Table~\ref{table1}.

Assume $f>1$. If $q=8$, then an inspection in~\cite{ATLAS} shows that there are no non-identity group elements $x$ and $y$ with $G={\bf N}_G(\langle x\rangle){\bf N}_G(\langle y\rangle)$. Suppose $q\neq 8$.
Let $s$ be a primitive prime divisor of $r^{2f}-1$ (such a prime exists by Zsigmondy's theorem~\cite{Zs} because we are excluding the case $q=8$).  Since $G=XY$, we see that either $X$ or $Y$ has order divisible by $s$. Replacing $X$ by $Y$ if necessary, we may assume that $s$ divides $|X|$. Using the subgroup structure of $\PGammaL_2(q)$ (see~\cite{BHRDbook}), we see that the only elements $x$  having normalizer of order divisible by $s$ are the elements lying in  a maximal torus $T$ of $\PGL_2(q)$ of order $q+1$. Now, for every $x\in T\setminus\{1\}$, we have $|\nor {\PGammaL_2(q)}{\langle x\rangle}|=2(q+1)f$ and
$$\mathrm{P}\Gamma\mathrm{L}_2(q)=\nor {\PGammaL_2(q)}{\langle x\rangle}\mathrm{PSL}_2(q).$$ Since $G=XY$, we get $$\mathrm{P}\Gamma\mathrm{L}_2(q)=\nor{\mathrm{P}\Gamma\mathrm{L}_2(q)}{\langle x\rangle}G=
\nor{\mathrm{P}\Gamma\mathrm{L}_2(q)}{\langle x\rangle}Y$$ and so we may assume that $G=\PGammaL_2(q)$. Thus $|X|=2(q+1)f$. As $|G||X\cap Y|=|X||Y|$, we obtain that $|G|/|X|=(q-1)q/2$ divides $|Y|$. Another inspection on the maximal subgroups of $\PGammaL_2(q)$ (for $q\neq 4,9$) shows that $\nor {\PGammaL_2(q)}{\langle y\rangle}$ is divisible by $(q-1)q/2 $ only when $q\in \{2^4,2^8\}$ and $y$ is a field automorphism of order $2$. If $L=\PSL_2(2^8)$ and $y$ is a field automorphism of order $2$, then $Y\cong \PGL_2(2^4).8$ and $|X\cap Y|=2$. However, $$\frac{|X||Y|}{|X\cap Y|}=\frac{2(2^8+1)8\cdot (2^8-1)2^7}{2}=\frac{|G|}{2}$$
and hence we have no examples when $q=2^8$. A computation with \textsc{Magma} shows that  the only factorization arising with $L=\PSL_2(16)$ is in Line~5 of Table~\ref{table1}.

It is easy to verify that in Line~3,~4 and~5 we have ${\bf C}_G(x){\bf C}_G(y)<G$ and hence there is no symbol $\surd$ in the $6^{\mathrm{th}}$ column (Line~5 can also be verified with an easy computer computation).
\end{proof}

Next we deal with  linear groups, where the primitive prime divisor $t_2$ in Lemma~\ref{hup} does not exist.

\begin{lemma}\label{casePSLnpart1}
 If $L=\PSL_3(q)$ and $q=r=2^\ell-1$, for some $\ell\in \mathbb{N}$, then $G\ne {\bf N}_G(\langle x\rangle){\bf N}_G(\langle y\rangle)$.
\end{lemma}
\begin{proof}
When $(n,q)=(3,3)$, the proof follows with a computer computation with the computer algebra system \textsc{Magma}. Lemma~\ref{le1} makes the search of factorizations $G=\nor G{\langle x\rangle}\nor G{\langle y\rangle}$ very efficient. In particular, for the rest of the argument, we suppose $q\ne 3.$

 Observe that $\ell>2$, because we are excluding the case $(n,q)=(3,3)$. Let $t_1$ be a primitive prime divisor of $r^3-1$ and let $T_1$ be a cyclic subgroup of $G$ of order $t_1$. As $t_1$ divides $|L|$, we have that $t_1$ divides $|X|$ or $|Y|$. Without loss of generality, we suppose that $t_1$ divides $|X|$ and hence $t_1$ divides $|\nor G{\langle x\rangle}|$. From Lemma~\ref{hup0}, replacing $T_1$ with a suitable conjugate, we have $T_1\leqslant {\bf C}_G(x)$. Therefore  $x\in \cent G{T_1}$. Let $T_x$ be the maximal torus of $\PGL_3(r)$ containing $T_1$. In particular, $T_x$ is a torus of order $(q^3-1)/(q-1)$. From Lemma~\ref{hup0}, we deduce $$\cent {\Aut(L)}{T_1}=T_x.$$ Thus $x\in T_x$ and hence $x$  is a semisimple element of order dividing $r^2+r+1$. Using this fact, we deduce that $\nor G{\langle x\rangle}$ has order a divisor of $6(r^2+r+1)$. Since this number is relatively prime to $r$, we deduce that $Y$ contains a Sylow $r$-subgroup $R$ of $G$. Thus $R\leqslant Y=\nor G{\langle y\rangle}$ and $R=\nor R{\langle y\rangle}$. A moment's thought gives that $y\in \Z R$ and hence $y$ is a transvection of $K$. Thus $|\nor {\PGL_3(r)}{\langle y\rangle}|=(r-1)^2r^3$ and hence $|Y|$ divides $2(r-1)^2r^3$. Therefore $|X||Y|$ divides $$
 6(r^2+r+1)
 \cdot 2(r-1)^2r^3=12r^3(r^3-1)(r-1)$$ and so does $|G|$, because $G=XY$. As $$|G|=|G:L||L|\geqslant \frac{r^3(r^3-1)(r^2-1)}{\gcd(3,r-1)},$$ we deduce $r+1=2^\ell$ divides $4$ and hence $\ell\leqslant 2$, contradicting the fact that $\ell>2$. 
\end{proof}

Next we deal with the exceptional case arising in part~\eqref{hup2}  of  Lemma~\ref{hup}.

\begin{lemma}\label{casePSLnpart3}
 If $L=\PSL_n(q)$, $q=r$, $n$ is a primitive prime divisor of $q^{n-1}-1$, $n\geqslant 3$ and ${\bf C}_{\mathrm{PSL}_n(q)}(x)$ or ${\bf C}_{\mathrm{PSL}_n(q)}(y)$  is a maximal torus of order $(q^n-1)/(q-1)$, then $G\ne {\bf N}_G(\langle x\rangle){\bf N}_G(\langle y\rangle)$.
\end{lemma}
\begin{proof}
Assume by contradiction that $G=XY$.
Here, $|\mathrm{Aut}(L):L|=2$. Without loss of generality, we may suppose that ${\bf C}_{\mathrm{PSL}_n(q)}(x)$  is a maximal torus of order $(q^n-1)/(q-1)$. Thus $|X|=n(q^n-1)/(q-1)$ when $G=L$, and $|X|=2n(q^n-1)/(q-1)$ when $G>L$. In both cases, $|G:X|=|L|/(n(q^n-1)/(q-1))$.

By consulting the maximal factorizations of almost simple groups with socle $\mathrm{PSL}_n(q)$ in~\cite{LPS1}, we deduce that $Y\leqslant P$, where $P\in \{P_1,P_{n-1}\}$ and $P_1,P_{n-1}$ are maximal parabolic subgroups. As $G=XY$, we deduce that $|Y|$ is divisible by $|G:X|=|L|/(n(q^n-1)/(q-1))$. As $|P|=|L|/((q^n-1)/(q-1))$ we have that $|P:Y|$ divides $n$. As $P$ has no subgroups having prime index $n$, we get $Y=P$. However, when $n\geqslant 3$, $P$ normalizes no cyclic non-identity subgroup.
\end{proof}

Finally, we deal with the general case.
\begin{lemma}\label{casePSLnpart2}
 If $L=\PSL_n(q)$, $n\geqslant 3$ and $(n,q)\ne (3,2)$, then $(G,x,y)$ is in Line~$6$ or~$7$ of Table~$\ref{table1}$.
\end{lemma}
\begin{proof}
When $n=3$ and $q=r=2^\ell-1$, for some $\ell\in \mathbb{N}$, the proof follows from Lemma~\ref{casePSLnpart1}. Similarly, when $q=r$, $n$ is a primitive prime divisor of $q^{n-1}-1$ and ${\bf C}_L(x)$ or ${\bf C}_L(y)$ is a maximal torus of order $(q^n-1)/(q-1)$, the proof follows from Lemma~\ref{casePSLnpart3}. Therefore, we exclude these cases from the rest of the proof of this lemma. 

When $(n,q)\in\{(3,4),(3,8),(4,4),(6,2),(7,2)\}$, the proof follows with a computer computation with the computer algebra system \textsc{Magma}. Lemma~\ref{le1} makes the search of factorizations $G=\nor G{\langle x\rangle}\nor G{\langle y\rangle}$ very efficient. In particular, for the rest of the argument, we suppose 
$$(n,q)\notin\{(3,4),(3,8),(4,4),(6,2),(7,2)\}.$$

For the rest of our argument, the primitive prime divisors $t_1$ and $t_2$ in Lemmas~\ref{hup0} and~\ref{hup} exist and moreover, part~\eqref{hup2} in Lemma~\ref{hup} does not arise.

Without loss of generality, we may suppose that $t_1$ divides $|X|$. Let $T_1$ be a cyclic subgroup of $X$ having order $t_1$ and set 
$$C_1:={\bf C}_{\mathrm{Aut}(L)}(T_1).$$
From Lemma~\ref{hup0}, we have
\begin{equation}\label{eq:C1}
x\in C_1=
\begin{cases}
\langle T_x,\iota\rangle&\iota \textrm{ graph automorphism of }L, \textrm{ if  }n \textrm{ is even},\\
\langle T_x, \iota\rangle&\iota \textrm{ graph-field automorphism of }L, \textrm{ if }n \textrm{ is odd and } f\textrm{ is even},\\
T_x&\textrm{if }nf\textrm{ is odd},
\end{cases}
\end{equation}
where $T_x$ is a maximal torus of $\mathrm{PGL}_n(q)$ having order $(q^n-1)/(q-1)$.
We now divide the rest of the argument in three cases, depending on whether $x\in T_x$ and on whether $n$ is even.

\subsection{Assume $x\in T_x$.}\label{subsection1} Thus $x$ is a semisimple element and $X$ is a field extension subgroup of $G$. Thus, using the information in~\cite[Table.B1,~B.2,~B.3]{BGbook}, we deduce that the order of $X$ divides 
\begin{equation}\label{eq:c19}\frac{|G:G\cap\mathrm{PGL}_n(q)|\ell}{q-1}|\mathrm{GL}_{n/\ell}(q^\ell)|,
\end{equation}
for some divisor $\ell$ of $n$ with $\ell>1$. 

We now turn our attention to $t_2$. From~\eqref{eq:c19}, we see that $t_2$ is relatively prime to $|X|$ and hence $t_2$ divides $|Y|$. 
Let $T_2$ be a cyclic subgroup of $Y$ having order $t_2$ and set 
$$C_2:={\bf C}_{\mathrm{Aut}(L)}(T_2).$$
From Lemma~\ref{hup}, we have
\begin{equation}\label{eq:C2}
y\in C_2=
\begin{cases}
\langle T_y,\iota\rangle&\iota \textrm{ graph automorphism of }L, \textrm{ if  }n \textrm{ is odd},\\
\langle T_y, \iota\rangle&\iota \textrm{ graph-field automorphism of }L, \textrm{ if }n \textrm{ and } f\textrm{ are even},\\
T_y&\textrm{if }(n-1)f \textrm{ is odd},
\end{cases}
\end{equation}
where $T_y$ is a maximal torus of $\mathrm{PGL}_n(q)$ having order $q^{n-1}-1$.

Assume $y\in T_y$. Thus the order of $Y$ divides 
\begin{equation}\label{eq:c191}|G:G\cap\mathrm{PGL}_n(q)|\kappa|\mathrm{GL}_{(n-1)/\kappa}(q^\kappa)|,
\end{equation}
for some divisor $\kappa$ of $n-1$. Assume first $\kappa>1$. Then, using~\eqref{eq:c19} and~\eqref{eq:c191}, we have
\begin{align*}
q^{\frac{n(n-1)}{2}}|G:G\cap\mathrm{PGL}_n(q)|_r&=&& |G|_r\leqslant |X|_r|Y|_r\\
&\leqslant&& |G:G\cap\mathrm{PGL}_n(q)|_r^2(\ell \kappa)_rq^{\frac{n}{2}\left(\frac{n}{\ell}-1\right)}q^{\frac{(n-1)}{2}\left(\frac{(n-1)}{\kappa}-1\right)}.
\end{align*}
Therefore
$$q^{\frac{n(n-1)}{2}}\leqslant |G:G\cap\mathrm{PGL}_n(q)|_r(\ell \kappa)_rq^{\frac{n}{2}\left(\frac{n}{\ell}-1\right)}q^{\frac{(n-1)}{2}\left(\frac{(n-1)}{\kappa}-1\right)}.$$
A computation yields that this inequality is satisfied only when $(n,q)=(3,2)$. The case $(n,q)=(3,2)$ is of no concern to us here, because we are excluding this case from the statement. Assume next $\kappa=1$. When $\kappa=1$, using the explicit structure of ${\bf C}_L(x)$ and ${\bf C}_L(y)$, we deduce $|X\cap Y|_r\geqslant |\mathrm{GL}_{n/\ell-1}(q^\ell)|_r$. Therefore, 
using~\eqref{eq:c19} and~\eqref{eq:c191}, we have
\begin{align*}
q^{\frac{n(n-1)}{2}}|G:G\cap\mathrm{PGL}_n(q)|_r&=&& |G|_r= \frac{|X|_r|Y|_r}{|X\cap Y|_r}\\
&\leqslant&& |G:G\cap\mathrm{PGL}_n(q)|_r^2\ell_rq^{\frac{n}{2}\left(\frac{n}{\ell}-1\right)}q^{\frac{(n-1)(n-2)}{2}}q^{-\frac{(n-\ell)}{2}\left(\frac{n}{\ell}-2\right)}.
\end{align*}
Therefore
$$q^{\frac{n(n-1)}{2}}\leqslant |G:G\cap\mathrm{PGL}_n(q)|_r\ell_rq^{\frac{n}{2}\left(\frac{n}{\ell}-1\right)}q^{\frac{(n-1)(n-2)}{2}}q^{-\frac{(n-\ell)}{2}\left(\frac{n}{\ell}-2\right)}.$$
A computation yields that this inequality is satisfied only when $q\in \{2,4\}$ and $\ell=2$.   The case $q=2$ is impossible because $\GL_{n-1}(2)$ is centerless, but $y\in {\bf Z}(\mathrm{GL}_{n-1}(q))$. When $q=4$, we deduce that $y$ is a semisimple element having an eigenspace of dimension $n-1$ and that $x$ is a semisimple element of order $5$ with no $1$-dimensional eigenspaces on $V$. In the case $(q,\ell)=(4,2)$, by refining the computation above comparing $|G|_2$ with $|X|_2|Y|_2/|X\cap Y|_2$, we deduce $G\nleq \langle\PGL_n(4),\tau\rangle$, where $\tau$ is the inverse-transpose graph automorphism. Thus we obtain the examples in Line~7 of Table~\ref{table1}, with the extra remark concerning $G$ in the fifth column.

Assume $y\notin T_y$. Suppose also that, for the time being, $y$ is an involution. Using~\cite[Table~B.1,~B.2,~B.3]{BGbook} (or~\cite[Chapter~4]{GLS}), we see that all involutions in $C_2\setminus T_y$ are $\Aut(L)$-conjugate to $\iota$. 
Thus $y$ is $\Aut(L)$-conjugate to a graph automorphism when $n$ is odd and $\iota$ is $\Aut(K)$-conjugate to a graph-field automorphism when $n$ and $f$ are even. Thus the order of $Y$ divides 
\begin{eqnarray}\label{eq:c192}
2f(q-1)|\mathrm{Sp}_{n-1}(q)|,&&\textrm{when }n \textrm{ is odd},\\\nonumber
2f(q-1)|\mathrm{GU}_{n}(q^{1/2})|,&&\textrm{when }n, f\textrm{ are even}.
\end{eqnarray}
Using~\eqref{eq:c19} and~\eqref{eq:c192}, we see with a computation similar (but simpler) to the one above that $|X|_r|Y|_r<|G|_r$ and hence this case does not arise. We give details only in the case that $Y$ is of type $\mathrm{Sp}_{n-1}(q)$. We have
\begin{align*}
|G|_r&=|XY|_r\leqslant |X|_r|Y|_r=(\ell|G:G\cap\mathrm{PGL}_n(q)|)_rq^{\frac{n}{2}\left(\frac{n}{\ell}-1\right)}\cdot (2f)_rq^{\frac{(n-1)^2}{4}}.
\end{align*}
Using $|G|_r=|G:G\cap\mathrm{PGL}_n(q)|_rq^{\frac{n(n-1)}{2}}$ and simplifying the expression above, we deduce
\begin{align*}
q^{\frac{n^2}{4}+\frac{n}{2}-\frac{1}{4}-\frac{n^2}{2\ell}}\leqslant (2\ell f)_r.
\end{align*}
Since $\ell>1$ and $n$ is odd, we have $n^2/2\ell\geqslant n^2/6$ and hence
$$q^{\frac{n^2}{12}+\frac{n}{2}-\frac{1}{4}}\leqslant (2\ell f)_r,$$
which is never satisfied.
 Suppose now that $y$ is not an involution. As $y\in C_2\setminus T_y=\langle T_y,\iota\rangle\setminus T_y$, we have $y^2\in T_y$, $y^2\ne 1$ and $G=X\nor G{\langle y^2\rangle}$. Therefore, we may apply the argument in the first part of the proof to the triple $(G,x,y^2)$ and we deduce that $(G,x,y^2)$ is in Line~7 of Table~\ref{table1}. Thus $n$ is even, $q=4$, $|x|=5$, $|y^2|=3$, $y^2$ is a semisimple element with an $(n-1)$-dimensional eigenspace and $x$ is a semisimple element with no eigenvalues in $\mathbb{F}_q$. Now, $|y|=6$ and $\cent L y\cong \SU_{n-1}(2)$. Then an easy computation, comparing $|G|_2$ with $|XY|_2$, yields $G\ne XY$; therefore, there are no further examples in this case.

\subsection{Assume $x\notin T_x$ and that $n$ is odd.}\label{subsection2} Then, from~\eqref{eq:C1}, $f$ is even. Suppose also that, for the time being, $x$ has order $2$.  Using again the information in~\cite[Section~$3.1$]{BGbook} (or in~\cite[Chapter~4]{GLS}), we see that $x$ is $\Aut(L)$-conjugate to $\iota$. Therefore  $x$ is $\Aut(L)$-conjugate to a graph-field automorphism. Hence  $${\bf }O^{r'}(X)\cong\mathrm{PSU}_n(q^{1/2}).$$ An inspection on the maximal factorizations in~\cite{LPS1,LPS2} reveals that there are no factorizations $G=XY$ with ${\bf }O^{r'}(X)\cong\mathrm{PSU}_n(q^{1/2})$. (This analysis could be omitted by comparing the size of a Sylow $r$-subgroup of $X$, $Y$ and $G$ in a fashion similar to the computations above.) Suppose now that $x$ is not an involution. Then $x^2\in T_x$ and $x^2\ne 1$. Therefore, we may apply \S~\ref{subsection1}  to the triple $(G,x^2,y)$ and we deduce that $G\ne \nor G{\langle x^2\rangle}\nor G y$, because except for $(n,q)=(3,2)$ there are no factorizations when $n$ is odd and the case $(n,q)=(3,2)$ does not occur here as we require $f$ even.

\subsection{Assume $x\notin T_x$, $|x|=2$ and $n$ is even.}\label{subsection3} Using~\cite[Section~$3.1$]{BGbook} (or~\cite[Chapter~4]{GLS}), we infer that, when $q$ is even, $C_1\setminus T_x$ contains a unique $\Aut(L)$-conjugacy class of involutions and, when $q$ is odd,  $C_1\setminus T_x$ contains two  $\Aut(L)$-conjugacy classes of involutions.  Then, using the information in~\cite[Section~$3.1$ and Table~B.3]{BGbook}, we obtain
\begin{equation}\label{double}
{\bf O}^{2'}(\cent{\mathrm{PGL}_n(q)}x)\cong
\begin{cases}
\mathrm{PSp}_n(q),&\textrm{or},\\
\mathrm{P}\Omega_{n}^-(q),&q \textrm{ odd}.
\end{cases}
\end{equation}
Then $t_2$ does not divide $|{\bf N}_G(\langle x\rangle)|$ and so $t_2$ divides $|Y|$. Thus $y$ is as in (\ref{eq:C2}).

Assume $y\in T_y$. We claim that,  using~\eqref{eq:c191}, the second possibility for $\cent Gx$ in~\eqref{double} does not give rise to any factorization $G=XY$. Indeed, we have
\begin{align*}
|G|_r&\leqslant |X|_r|Y|_r\le(|G:\mathrm{PGL}_n(q)|)_rq^{\frac{n(n-2)}{4}}\cdot |G:G\cap\mathrm{PGL}_n(q)|_r\kappa_rq^{\frac{(n-1)}{2}\left(\frac{n-1}{\kappa}-1\right)}.
\end{align*}
When $\kappa>1$, rearranging the terms and using the fact that $(n-1)/\kappa\leqslant (n-1)/3$,  we deduce 
$$q^{\frac{n^2}{12}+\frac{5n}{6}-\frac{2}{3}}\leqslant (|G:G\cap\mathrm{PGL}_n(q)|\kappa)_r,$$
which is impossible. When $\kappa=1$, a similar computation taking in account that $|X\cap Y|_r\geqslant |X\cap Y\cap L|_r\geqslant |\Omega_{n-1}(q)|_r=q^{(n-1)^2/4}$ yields another contradiction.
Thus ${\bf O}^{2'}(\cent{\mathrm{PGL}_n(q)}x)\cong\mathrm{PSp}_n(q)$.
 Using the formula for $|\PSp_n(q)|$, we obtain
$$q^{\frac{n(n-1)}{2}}\leqslant |G|_r\leqslant |X|_r|Y|_r\leqslant (2f)_r^2q^{\frac{n^2}{4}}q^{\frac{n-1}{2\kappa}\left(\frac{n}{\kappa}-1\right)}.$$
This inequality is satisfied only when $\kappa=1$. Thus $y$ is a semisimple element having an eigenspace of dimension $n-1$. The examples arising in this case are in Line~6 of Table~\ref{table1}. 

Assume $y\notin T_y$ and $y$ is an involution. In particular, from~\eqref{eq:C2}, as $n$ is even, $f$ is also even and $y$ is $\Aut(L)$-conjugate to a graph-field automorphism. Therefore ${\bf O}^{r'}(\cent Ly)\cong \mathrm{PSU}_{n-1}(q^{1/2})$. Using the formulae for $|\PSU_{n-1}(q^{1/2})|$ and for $|\PSp_n(q)|$ (or $\mathrm{P}\Omega_{n}^-(q)$, depending on the $\Aut(L)$-conjugacy class of $x$) and taking in account whether $r$ is odd or $r=2$, we see with a computation that $|X|_r|Y|_r<|G|_r$. Therefore this case does not arise. 

Assume $y\notin T_y$ and $y$ is not an involution. As $y\in C_2\setminus T_y=\langle T_y,\iota\rangle\setminus T_y$, we have $y^2\in T_y$, $y^2\ne 1$ and $G=X\nor G{\langle y^2\rangle}$. Therefore, we may apply the argument in the first part to the triple $(G,x,y^2)$ and we deduce that $(G,x,y^2)$ is in Line~6 of Table~\ref{table1}. Thus $|y^2|$ divides $q-1$ and $y^2$ is a semisimple element with an $(n-1)$-dimensional eigenspace. Assume that $|y|$ is divisible by some odd prime $p$. Replacing $y$ by $y^{|y|/2p}$ if necessary, we may suppose that $|y|=2p$. Now, $y^p\in C_2\setminus T_y=\langle T_y,\iota\rangle\setminus T_y$, $y^p$ is an involution and $G=X\nor G{\langle y^p\rangle}$; however, we have shown in the previous paragraph that this is impossible. Therefore, $y$ has order a power of $2$ and hence, replacing $y$ by $y^{|y|/4}$ if necessary, we may suppose that $|y|=4$. In particular, $q$ is odd, because when $q$ is even $T_y$ has odd order. To conclude the rest of our analysis we identify $T_y$ with the multiplicative group of $\mathbb{F}_{q^{n-1}}$. (In particular, under this identification, we refer to an element of $\mathbb{F}_{q^{n-1}}^\ast$ as an element of $\mathrm{PGL}_n(q)$.) Let $\lambda$ be a generator of $\mathbb{F}_{q^{n-1}}^\ast$. Then $y=\lambda^\ell \iota$, where $\ell$ is a divisor of $q^{n-1}-1$ and $\iota$ is a graph-field automorphism. Now,
$$y^2=(\lambda^\ell \iota)^2=\lambda^\ell(\lambda^{\ell})^\iota=\lambda^\ell\lambda^{-\ell q^{1/2}}=\lambda^{\ell(1-q^{1/2})}.$$
As $y^2$ has order $2$, we deduce $\ell(1-q^{1/2})=\kappa (q^{n-1}-1)/2$, for some $\kappa\in\mathbb{N}$. Thus $$\ell=\kappa \frac{q^{n-1}-1}{2(1-q^{1/2})}.$$
This shows that $\lambda^\ell$ has order a divisor of $2(q^{1/2}-1)$ and hence $\lambda^\ell\in \mathbb{F}_q^\ast$. Since all elements of $\mathrm{PGL}_n(q)$ in $\mathbb{F}_q^\ast\subseteq \mathbb{F}_{q^{n-1}}^\ast$  have the same centralizer, we have
\begin{align*}
{\bf C}_{\mathrm{PGL}_n(q)}(y^2)&=
{\bf C}_{\mathrm{PGL}_n(q)}(\lambda^\ell)
\end{align*}
and hence
\begin{align*}
{\bf C}_{\mathrm{PGL}_n(q)}(y)&=
{\bf C}_{\mathrm{PGL}_n(q)}(\langle \lambda^\ell ,\iota\rangle).
\end{align*}
Now, ${\bf C}_{\mathrm{PGL}_n(q)}(\lambda^\ell)\cong \mathrm{GL}_{n-1}(q)$, from which we deduce that $${\bf C}_{\mathrm{PGL}_n(q)}(y)\cong {\bf C}_{\mathrm{GL}_{n-1}(q)}(\iota)\cong \mathrm{GU}_{n-1}(q^{1/2}).$$
Using this explicit description of ${\bf C}_{\mathrm{PGL}_n(q)}(y)$ it is not hard to verify that $|X|_r|Y|_r<|G|_r$.

\subsection{Assume $x\notin T_x$, $|x|>2$ and $n$ is even.}\label{subsection4}Here $x^2\in T_x$ and $x^2\ne 1$. Applying \S~\ref{subsection1} to the triple $(G,x^2,y)$, we deduce that $(G,x^2,y)$ is in Line~7 of Table~\ref{table1}. Thus $n$ is even, $q=4$, $|x^2|=5$, $|y|=3$, $x^5$ has no eigenvalue in $\mathbb{F}_q$ and $y$ has an $(n-1)$-dimensional eigenspace on $V$. Thus $|x|=10$ and $\cent L x\cong \GU_{n/2}(4)$. However, it is not hard to see that $|G|_2\ne |X|_2|Y|_2/|X\cap Y|_2$. Therefore, no further example arises.

\smallskip Using the explicit description of ${\bf N}_G(\langle x\rangle )$, ${\bf N}_G(\langle y\rangle)$ in Line~6 and~7, it is readily seen that $G={\bf C}_G(x){\bf C}_G(y)$ when $(G,x,y)$ is in Line~6 and  ${\bf C}_G(x){\bf C}_G(y)<G$ when $(G,x,y)$ is in Line~7.  Thus we have the $\surd$ symbol in Line~6, whereas $\surd$ is omitted in Line~7.
\end{proof}

\section{Classical groups: unitary groups}\label{sec:classicalpsu}

In this section we assume that $G$ is an almost simple group with socle $L=\mathrm{PSU}_n(q).$ Exactly as in Section~\ref{sec:classicalpsl}, we start with three technical lemmas which help to locate the elements $x$ and $y$ with $G={\bf N}_G(\langle x\rangle){\bf N}_G(\langle y\rangle)$.

\begin{lemma}\label{hup0PSU}Let $n$ be even. Suppose $r^{2f(n-1)}-1$ admits a primitive prime divisor $t_2$. Let $g\in \mathrm{Aut}(\mathrm{PSU}_n(q))$ with $t_2$ dividing $|{\bf N}_{\mathrm{Aut}(\mathrm{PSU}_n(q))}{(\langle g\rangle)}|$ and let $T_2$ be a cyclic subgroup of order $t_2$ in ${\bf N}_{\mathrm{Aut}(\mathrm{PSU}_n(q))}{(\langle g\rangle)}$. Then 
\[
g\in {\bf C}_{\mathrm{Aut}(\mathrm{PSU}_n(q))}(T_2)= 
T,
\]
where $T$ is a maximal torus of $\mathrm{PGU}_n(q)$ having order $q^{n-1}+1$. In particular, $|\mathrm{Aut}(\mathrm{PSU}_n(q)):{\bf C}_{\mathrm{Aut}(\mathrm{PSU}_n(q))}(g)|$ is relatively prime to $t_2$.
\end{lemma}
\begin{proof}
Let $T$ be a maximal torus of $\mathrm{PGU}_n(q)$ containing $T_2$. Observe that $T_2$ in its action on $V=\mathbb{F}_{q^2}^n$ fixes a $1$-dimensional non-degenerate subspace and acts irreducibly on its complement. Thus $|T|=q^{n-1}+1$. Using the information in~\cite[Chapter~4]{GLS} and~\cite[Chapter~$3.2$, Table~B.1,~B.2,~B.3]{BGbook} together with the fact that $n$ is even, we obtain
\[{\bf C}_{\mathrm{Aut}(\mathrm{PSU}_n(q))}(T_2)=T.
\]
If $T_2\leqslant {\bf C}_{\mathrm{Aut}(\mathrm{PSU}_n(q))}{( g)}$, then $g\in {\bf C}_{\mathrm{Aut}(\mathrm{PSU}_n(q))}{(T_2)}=T$ and hence $|\mathrm{Aut}(\mathrm{PSU}_n(q)):{\bf C}_{\mathrm{Aut}(\mathrm{PSU}_n(q))}(g)|$ is relatively prime to $t_2$, because so is $|\mathrm{Aut}(\mathrm{PSU}_n(q)):T|$.  

It remains to consider the case that $t_2$ does not divide the order of ${\bf C}_{\mathrm{Aut}(\mathrm{PSU}_n(q))}{( g)}$. We aim to prove that this case cannot arise. As $t_2$ divides ${\bf N}_{\mathrm{Aut}(\mathrm{PSU}_n(q))}{(\langle g\rangle)}$ and as ${\bf N}_{\mathrm{Aut}(\mathrm{PSU}_n(q))}{(\langle g\rangle)}/{\bf C}_{\mathrm{Aut}(\mathrm{PSU}_n(q))}{(g)}$ acts faithfully as a group of automorphisms on the cyclic group $\langle g\rangle$, we deduce that $t_2$ divides $\varphi(|g|)$, where $\varphi$ is the Euler's totient function. In particular, $|g|$ is divisible by a prime $p$ with $t_2\mid p-1$. Without loss of generality, we may suppose that $|g|=p$. 
As $2f(n-1)$ divides $t_2-1$, we have $2f(n-1)<t_2$ and hence $g\in\mathrm{PSU}_n(q)$. Now, as $T_2$ acts non-trivially on $\langle g\rangle$, we deduce that $T_2$ permutes non-trivially the eigenspaces of $g$. However, this is a contradiction because $t_2>2f(n-1)\geqslant n$.
\end{proof}

\begin{lemma}\label{hupPSU}Let $n$ be a positive integer with $n/2$ even, let $t_1$ be a primitive prime divisor of $r^{fn}-1$, let $g\in \mathrm{Aut}(\mathrm{PSU}_n(q))$ with $t_1$ dividing $|{\bf N}_{\mathrm{Aut}(\mathrm{PSU}_n(q))}{(\langle g\rangle)}|$ and let $T_1$ be a cyclic subgroup of order $t_1$ in ${\bf N}_{\mathrm{Aut}(\mathrm{PSU}_n(q))}{(\langle g\rangle)}$. Then
\[
g\in {\bf C}_{\mathrm{Aut}(\mathrm{PSU}_n(q))}(T_1)=
\langle T,\iota\rangle,
\] 
where $T$ is a maximal torus of $\mathrm{PGU}_n(q)$ having order $(q^{n}-1)/(q+1)$, $\iota\in \mathrm{Aut}(\mathrm{PSU}_n(q))$, $|\iota|=2$ and ${\bf C}_{\mathrm{PGU}_n(q)}(\iota)\cong \mathrm{PGSp}_n(q)$. In particular, $|{\bf N}_{\mathrm{Aut}(\mathrm{PSU}_n(q))}{(\langle g\rangle)}:{\bf C}_{\mathrm{Aut}(\mathrm{PSU}_n(q))}(g)|$ is relatively prime to $t_1$.
\end{lemma}
\begin{proof}
The hypothesis $n/2$ even is used to guarantee that $t_1$ divides $q^i-(-1)^i$, only when $i=n$. The rest of the proof follows verbatim the proofs of Lemmas~\ref{hup0},~\ref{hup} and~\ref{hup0PSU}. The structure of ${\bf C}_{\mathrm{Aut}(\mathrm{PSU}_n(q))}(T_1)$ can be inferred from~\cite[Section~3.2, Table~B.4]{BGbook} or~\cite[Chapter~4]{GLS}.
\end{proof}

\begin{lemma}\label{hupPSU28}Let $n\geqslant 4$ be even with $n/2$ odd. Suppose $r^{fn/2}-1$ admits a primitive prime divisor $t_1'$. Let $g\in \mathrm{Aut}(\mathrm{PSU}_n(q))$ with $t_1'$ dividing $|{\bf N}_{\mathrm{Aut}(\mathrm{PSU}_n(q))}{(\langle g\rangle)}|$ and let $T_1'$ be a cyclic subgroup of order $t_1'$ in ${\bf N}_{\mathrm{Aut}(\mathrm{PSU}_n(q))}{(\langle g\rangle)}$. Then
\[
g\in {\bf C}_{\mathrm{Aut}(\mathrm{PSU}_n(q))}(T_1')=
\langle T,\iota\rangle,
\] 
where $T$ is a maximal torus of $\mathrm{PGU}_n(q)$ having order $(q^{n}-1)/(q+1)$, $\iota\in \mathrm{Aut}(\mathrm{PSU}_n(q))$, $|\iota|=2$ and ${\bf C}_{\mathrm{PGU}_n(q)}(\iota)\cong \mathrm{PGSp}_n(q)$. In particular, $|{\bf N}_{\mathrm{Aut}(\mathrm{PSU}_n(q))}{(\langle g\rangle)}:{\bf C}_{\mathrm{Aut}(\mathrm{PSU}_n(q))}(g)|$ is relatively prime to $t_1'$.
\end{lemma}
\begin{proof}
The hypothesis $n/2$ odd is used to guarantee that $t_1'$ divides $q^i-(-1)^i$, only when $i=n$. The rest of the proof follows verbatim the other analogous proofs. For this lemma, the only part that is not trivial is  showing that, $t_1'$ is relatively prime to $|{\bf N}_{\mathrm{Aut}(\mathrm{PSU}_n(q))}(\langle g\rangle):{\bf C}_{\mathrm{Aut}(\mathrm{PSU}_n(q))}(g)|$. Arguing as usual, we have that $t_1'\mid p-1$, $|g|=p$ and $g$ is semisimple because $fn/2<t_1'$. As $t_1'$ is a primitive prime divisor of $r^{fn/2}-1$ we have that  $fn/2$ divides $t_1'-1$ and hence $$\alpha f\frac{n}{2}+1=t_1',$$ for some $\alpha\in\mathbb{N}$. If $\alpha f>1$, then $t_1'>n$. If $\alpha=f=1$, then $t_1'=n/2+1$ is even because $n/2$ is odd, which is a contradiction because $n\geqslant 4$. Therefore, in all cases $t_1'>n$.
Now, $T_1'$ permutes the eigenspaces of $g$, but it must permute the eigenspaces trivially because $t_1'>n$.
\end{proof}

From here onward we make use more intensively of the work of Liebeck, Praeger and Saxl on maximal factorizations~\cite{LPS1,LPS2}.
\begin{lemma}\label{casePSU}
 If $L=\PSU_n(q)$ with $n\ge3$, then $(G,x,y)$ is in Line~$8$ or~$9$ of Table~$\ref{table1}$.
\end{lemma}
\begin{proof}
When $(n,q)\in\{(3,3),(3,5),(3,8),(4,2),(4,3),(6,2),(6,4),(9,2),(12,2)\}$, the proof follows with a computer computation with the computer algebra system \textsc{Magma}.

 From~\cite{LPS1,LPS2}, we see that, when $n$ is odd, $G$ admits no proper factorization, except when $(n,q)\in \{(3,3),(3,5),(3,8),(9,2)\}$. In particular, since we have already dealt with these cases, for the rest of the proof, we may assume that $n$ is even. Set $m:=n/2$. The other pairs that we have excluded with the computer computation allow us  to conclude that the only maximal factorizations of $G$ are listed in~\cite[Table~1]{LPS1} and, via Zsigmondy's theorem, to guarantee the existence of a primitive prime divisor of $r^{2f(n-1)}-1$ and  $r^{fn}-1$ when $n/2=m$ is even, and $r^{fn/2}-1$ when $n/2=m$ is odd. 

Let $t_2$ be a primitive prime divisor of $r^{2f(n-1)}-1$. As $t_2$ divides $|L|$, without loss of generality, we may suppose that $t_2$ divides $|Y|$. Let $T_2$ be a cyclic subgroup of order $t_2$ in $Y$ and let $T_y$ be a maximal torus of $\PGU_n(q)$ containing $T_2$. From Lemma~\ref{hup0PSU}, we obtain
\[
y\in T_y,
\]
where $T_y$ is a torus having order $q^{n-1}+1$.
 As $n$ is even, from~\cite[Section~$3.2$]{BGbook}, the order of $Y$ divides 
\begin{equation}\label{eq:c19u1}|G:G\cap \mathrm{PGU}_n(q)|\kappa|\mathrm{GU}_{(n-1)/\kappa}(q^\kappa)|,
\end{equation}
for some divisor $\kappa$ of $n-1$. Moreover, from the structure of $T_y$, we also deduce that $Y$ is contained in the stabilizer of a $1$-dimensional non-degenerate subspace of $V$.
We claim that
\begin{align}\label{kappa=1PSU}
\kappa=1.
\end{align}

Since we have excluded above the pairs $(n,q)\in \{(4,2),(4,3),(6,2),(12,2)\}$ and since $Y$ fixes a non-degenerate $1$-dimensional subspace, using the classification of the factorizations of almost simple groups with socle $L=\PSU_n(q)$ in~\cite{LPS1,LPS2}, we see that $X$ must be contained in a maximal subgroup $A$ of $G$ with the property that one of the following holds:
\begin{enumerate}[label=(\roman*)]
\item\label{casepsu1}$L\cap A=P_m$,
\item\label{casepsu2}$L\cap A=\mathrm{PSp}_{2m}(q)=\mathrm{PSp}_n(q)$,
\item\label{casepsu3}$L\cap A={\hat{}\,}\SL_{m}(4).2$, $q=2$, $m\geqslant 3$,
\item\label{casepsu4}$L\cap A={\hat{}\,}\SL_m(16).3.2$, $q=4$ and $G\geqslant L.4$.
\end{enumerate}
In Cases~\ref{casepsu2},~\ref{casepsu3} and~\ref{casepsu4}, by comparing the size of a Sylow $r$-subgroup of $Y$ with a Sylow $r$-subgroup of $A$, we deduce that $\kappa=1$, that is,~\eqref{kappa=1PSU} holds true. All of these computations are straightforward and we only give details to the case~\ref{casepsu2}. Here,
\begin{align*}
|G|_r&=|G:G\cap\mathrm{PSU}_n(q)|_rq^{\frac{n(n-1)}{2}},\\
|X|_r&\leqslant (2f)_rq^{(n/2)^2},\\
|Y|_r&= |G:G\cap\mathrm{PSU}_n(q)|_r\kappa_rq^{\frac{(n-\kappa)\left(\frac{n}{\kappa}-1\right)}{2}}.
\end{align*}
Now, a tedious computation using this information yields that
$|G|_r\leqslant |X|_r|Y|_r$ is satisfied only when $\kappa=1$.

 In Case~\ref{casepsu1}, using the structure of $P_m$ (which can be deduced by fixing a hyperbolic basis for $V$), we have $$P_m\cap L\cong E_q^{m^2}:\mathrm{SL}_m(q^2).(q-1).$$
We claim that also in this case $\kappa=1$. To prove this claim we use the factorization of the order of $|\mathrm{SL}_m(q^2)|$ into cyclotomic polynomials,~\eqref{eq:c19u1} and Zsigmondy's theorem.
Assume first that $n\geqslant 8$ and let $t'$ be a primitive prime divisor of $q^{2(n-3)}-1$. In particular, $t'$ divides $q^{n-3}+1$ and hence $t'$ divides $|L|$, because $n$ is even. Now, as $2(n-3)>n$, $t'$ is relatively prime to $|P_m|$ because $$(q^{2m}-1)(q^{2m-2}-1)\cdots (q^4-1)(q^2-1)$$ cannot be divisible by $t'$. Therefore, $t'$ divides $|Y|$; but this is only possible when $\kappa=1$.  When $n=4$, we have $n-1=3$ and hence $\kappa\in \{1,3\}$. However, if $\kappa=3$, then $|Y|\leqslant 3(q^3+1)|G:G\cap\mathrm{PGU}_4(q)|$ and hence $$|G:X|\leqslant 3(q^3+1).$$ The minimal degree of a faithful permutation representation of $\mathrm{PSU}_4(q)$ is $(q+1)(q^3+1)$ (see~\cite[Table~4]{Guest}). As $3(q^3+1)<(q+1)(q^3+1)$, we have $L\leqslant X$, which is a contradiction. Finally, when $n=6$, we have $n-1=5$ and hence $\kappa\in \{1,5\}$. However, if $\kappa=5$, then $|Y|\leqslant 5(q^5+1)|G:G\cap\mathrm{PGU}_6(q)|$ and hence $$|G:X|\leqslant 5(q^5+1).$$ The minimal degree of a faithful permutation representation of $\mathrm{PSU}_6(q)$ is at least  $q^{5}(q^4+q^2+1)$ (see~\cite[Table~4]{Guest}). As $5(q^5+1)<q^5(q^4+q^2+1)$,  we have $L\leqslant X$, which is a contradiction. This concludes the proof of our claim and hence we have proved~\eqref{kappa=1PSU}.

From~\eqref{kappa=1PSU}, $Y\cap L$ is of type 
${\hat{}\,}\mathrm{GU}_{n-1}(q)$ and $y\in \Z{\,{\hat{}\,}\mathrm{GU}_{n-1}(q)}$. Thus $y$ is a semisimple element of order a divisor of $q+1$ and $y$ has an $(n-1)$-dimensional eigenspace. In particular, the order of $Y$ divides
\begin{equation}\label{eq:true}|G:G\cap \mathrm{PGU}_n(q)||\mathrm{GU}_{n-1}(q)|.\end{equation}

\smallskip

Using the description of $Y$, we have that $|G:Y|=|X:X\cap Y|$ is divisible by $$q^{n-1}\cdot\frac{q^{n}-1}{\gcd(n,q+1)(q+1)}.$$ Let $t_1$ be a primitive prime divisor of $r^{fn}-1=q^n-1$ when $n/2$ is even and let $t_1$ be a primitive prime divisor of $r^{fn/2}-1=q^{n/2}-1$ when $n/2$ is odd. Observe that in either case, $t_1$ divides $|X|$. Let $T_1$  be a 
cyclic subgroup of $X$ of order  $t_1$, let 
$C_1:={\bf C}_{\mathrm{Aut}(L)}(T_1)$ and  let $T_x$ be a maximal torus of $\PGU_n(q)$ 
containing $T_1$. From Lemmas~\ref{hupPSU} and~\ref{hupPSU28}, we obtain
\[
C_1=
\langle T_x,\iota\rangle,\quad|\iota|=2, \, \iota\in \Aut(\mathrm{PSU}_n(q))\setminus\mathrm{PSU}_n(q)\textrm{ and }\cent{\mathrm{PGU}_n(q)}\iota\cong\mathrm{PGSp}_n(q).
\]

\subsection{Assume $x\in T_x$.}\label{subsection21} Thus $x$ is a semisimple element and $X$ is a field extension subgroup of $G$.  Thus, using the information in~\cite[Section~$3.2$, Table~B.1,~B.2,~B.3]{BGbook}, we deduce that  the order of $X$ divides 
\begin{equation}\label{eq:c19u}
\frac{|G:G\cap \mathrm{PGU}_n(q)|}{q+1}\ell|\mathrm{GU}_{n/\ell}(q^\ell)| \quad\textrm{ or }\quad \frac{|G:G\cap \mathrm{PGU}_n(q)|}{q+1}\ell|\mathrm{GL}_{n/\ell}(q^\ell)|,
\end{equation}
for some divisor $\ell$ of $n$ with $\ell\geqslant 2$: where the case on the right occurs when $\ell$ is even and the case on the left occurs when $\ell$ is odd.

From the structure of $Y$, we deduce $|X\cap Y|_r\geqslant |\mathrm{GU}_{n/\ell-1}(q^\ell)|_r$ if $\ell$ is odd, and $|X\cap Y|_r\geqslant |\mathrm{GL}_{n/\ell-1}(q^\ell)|_r$ if $\ell$ is even. Suppose first that $\ell$ is odd. 
Using~\eqref{eq:true} and~\eqref{eq:c19u}, we have
\begin{align*}
|G:G\cap \mathrm{PGU}_n(q)|_rq^{\frac{n(n-1)}{2}}&=&& |G|_r\leqslant \frac{|X|_r|Y|_r}{|X\cap Y|_r}\\
&\leqslant &&|G:G\cap \mathrm{PGU}_n(q)|_r^2\ell_rq^{\frac{n}{2}\left(\frac{n}{\ell}-1\right)}q^{\frac{(n-1)(n-2)}{2}}q^{-\frac{(n-\ell)}{2}\left(\frac{n}{\ell}-2\right)}.
\end{align*}
A computation yields that this inequality is never satisfied since $\ell\geqslant 3$. This shows that $\ell$ is even and $X$ is of type $\mathrm{GL}_{n/\ell}(q^\ell).\ell$. Then, using~\eqref{eq:true} and~\eqref{eq:c19u}, we have again
\begin{equation}\label{eq:true1}q^{\frac{n(n-1)}{2}}\leqslant (|G:G\cap \mathrm{PGU}_n(q)|\ell)_r
q^{\frac{n}{2}\left(\frac{n}{\ell}-1\right)}q^{\frac{(n-1)(n-2)}{2}}q^{-\frac{(n-\ell)}{2}\left(\frac{n}{\ell}-2\right)}.
\end{equation}
Another computation in the same spirit as the one above shows that $\ell=2$ and $q\in \{2,4,16\}$. By refining the computation above, we see that the case $q=16$ does not actually arise. This can be seen by observing that~\eqref{eq:true1} is satisfied only when $G$ contains the whole field automorphism of $\mathbb{F}_{q^2}$, but now we can refine the bound $|X\cap Y|_2$ with $2|\mathrm{GL}_{n/2-1}(q^2)|_2$. Now, another computation with this refined value for $|X\cap Y|_2$ excludes the case $q=16$.   (Observe that this is in line with~\cite[Table~1]{LPS1}, where we see that maximal factorizations using $\mathcal{C}_3$-subgroups arise only when $q\in \{2,4\}$: see Cases~\ref{casepsu3} and~\ref{casepsu4} above.) When, $q=2$, $X$ is not a local subgroup (see for instance~\cite[Proposition~4.2.4]{KL}) and hence this case does not arise. Thus $q=4$, $x$ is an element of order $(q^2-1)/(q+1)=15/5=3$ 
having no eigenvalue in $\mathbb{F}_{q^2}$ and $y$ is an element of order $5$ having an eigenspace of dimension $n-1$. Moreover, consulting~\cite[Table~1]{LPS1} or by a careful computation as above, we see that $4$ divides $|G:L|_2$. This example is in Line~8 of Table~\ref{table1}.

\subsection{Assume $x\notin T_x$ and $|x|=2$.}\label{subsection22}  Using the results in~\cite[Table~B.4]{BGbook} (or in~\cite[Chapter~4]{GLS}), we see that ${\bf O}^{r'}(\cent L x)$ is of type $$\mathrm{PSp}_n(q),\, \hbox{ or }\mathrm{P}\Omega_n^\varepsilon(q)$$
and the second case only occurs when $q$ is odd.  By distinguishing these two possibilities for $X$, it is not hard to prove that $|G|_r|X\cap Y|_r=|X|_r|Y|_r$ yields that $X$ is of type $\PSU_n(q^{1/2})$ or $\PSp_n(q)$ (to exclude the case $\POmega_n^\varepsilon(q)$ we require the fact that $q$ is odd). Incidentally, the case that $X$ is of type $\POmega_n^\varepsilon(q)$ can also be excluded by checking~\cite[Table~1]{LPS1}.

When $X$ is of type $\PSp_n(q)$, we obtain that $x$ is an involution of order $2$ with $\cent Lx\cong \PSp_n(q)$ and that $y$ is an element of order dividing $q+1$ and  having an eigenspace of dimension $n-1$.  This example is in Line~9 of Table~\ref{table1}.


\subsection{Assume $x\notin T_x$ and $|x|>2$.}\label{subsection23} As $x^2\in T_x\setminus\{1\}$ and $x^2\ne 1$, applying \S\ref{subsection21} to the factorization $G=\nor G{\langle x^2\rangle}Y$, we obtain that the triple $(G,x^2,y)$ is  in Line~$8$ of Table~\ref{table1}. Therefore, $q=4$, $\nor L{\langle x^2\rangle}$ is of type $\GL_{n/2}(q^2).2$ and $x^2$ has order $3$. Now, $x^3\notin T_x$ and $x^3$ has order $2$. Therefore, applying \S\ref{subsection22} to the factorization $G=\nor G{\langle x^3\rangle}Y$, we obtain that the triple $(G,x^3,y)$ is in Line~9 of Table~\ref{table1}. Thus $x^2\in {\bf N}_L(\langle x^3\rangle)=\cent L{x^3}\cong \mathrm{Sp}_n(4)$. Thus 
\begin{equation}\label{neverend}{\bf N}_L(\langle x\rangle )={\bf N}_L(\langle x^2\rangle)\cap{\bf C}_L(x^3)\cong{\bf N}_{\mathrm{Sp}_n(4)}(\langle x^2\rangle)\cong 3.\mathrm{SL}_{n/2}(4).2.
\end{equation}

Now, let $p$ be a primitive prime divisor of $2^{fn}-1=q^{n}-1$. Observe that, from~\eqref{neverend}, $p$ is relatively prime to $|X|$. When $n/2$ is even, $p$ is also relatively prime to $|\mathrm{GU}_{n-1}(q)|$ and hence $p$ is relatively prime to $|Y|$. Therefore, in order to have $G=XY$,  $n/2$ must be odd. Assume $n/2$ is odd. Given $i\in \{1,\ldots,n\}$, $p$ divides $q^i-(-1)^i$ if and only if $i\in \{n/2,n\}$. 
 Therefore, $|G|_p>|Y|_p$ and hence also in this case we have $G\ne XY$. Summing up, we have shown that this case does not give rise to a factorization of $G$.
 
\smallskip
 
 Using the explicit description of ${\bf N}_G(\langle x\rangle )$, ${\bf N}_G(\langle y\rangle)$ in Line~8 and~9, it is readily seen that $G={\bf C}_G(x){\bf C}_G(y)$ when $(G,x,y)$ is in Line~9 and  ${\bf C}_G(x){\bf C}_G(y)<G$ when $(G,x,y)$ is in Line~8.  Thus we have the $\surd$ symbol in Line~9, whereas $\surd$ is omitted in Line~8.
\end{proof}


\section{Classical groups: symplectic groups}\label{sec:classicalpsp}
\begin{lemma}\label{hup0PSp}Let $n\geqslant 4$ be even. Suppose  $r^{fn}-1$ admits a primitive prime divisor $t_1$. Let $g\in \mathrm{Aut}(\mathrm{PSp}_n(q))$ with $t_1$ dividing $|{\bf N}_{\mathrm{Aut}(\mathrm{PSp}_n(q))}{(\langle g\rangle)}|$ and let $T_1$ be a cyclic subgroup of order $t_1$ in 
${\bf N}_{\mathrm{Aut}(\mathrm{PSp}_n(q))}{(\langle g\rangle)}$. Then 
\[
g\in {\bf C}_{\mathrm{Aut}(\mathrm{PSp}_n(q))}(T_1)=
\begin{cases}
\langle T,\iota\rangle,&\textrm{when }n=4, r=2,f \textrm{ is odd},\\&\iota\textrm{ graph-field automorphism,}\\
T,&\textrm{otherwise},
\end{cases}
\]
where $T$ is a maximal torus of $\mathrm{PGSp}_n(q)$ having order $q^{n/2}+1$, that is, $T$ is a Singer cycle. In particular, $|\mathrm{Aut}(\mathrm{PSp}_n(q)):{\bf C}_{\mathrm{Aut}(\mathrm{PSp}_n(q))}(g)|$ is relatively prime to $t_1$.
\end{lemma}

\begin{proof}
Let $T$ be a maximal torus of $\mathrm{PGSp}_n(q)$ containing $T_1$. Then $T$ is a cyclic group of order $q^{n/2}+1$. The structure of ${\bf C}_{\mathrm{Aut}(\mathrm{PSp}_n(q))}(T_1)$ follows consulting~\cite[Section~3.4]{BGbook} and~\cite[Chapter~4]{GLS}.

Suppose that $T_1\leqslant {\bf C}_{\mathrm{Aut}(\mathrm{PSp}_n(q))}(g)$. 
Then $g\in {\bf C}_{\mathrm{Aut}(\mathrm{PSp}_n(q))}(T_1)$. If $g\in T$, then 
$|\mathrm{Aut}(\mathrm{PSp}_n(q)):{\bf C}_{\mathrm{Aut}(\mathrm{PSL}_n(q))}(g)|$ 
is relatively prime to $t_1$, because 
$T\leqslant {\bf C}_{\mathrm{Aut}(\mathrm{PSp}_n(q))}(g)$ 
and $|\mathrm{Aut}(\mathrm{PSp}_n(q)):T|$ is relatively prime to $t_1$. If $g\notin T$,  then $n=4$, $r=2$ and $f$ is odd. Moreover, $g$ is conjugate to $\iota$ because $T$ has odd order. Therefore, $g$ is an involution and hence 
${\bf C}_{\mathrm{Aut}(\mathrm{PSp}_n(q))}(g)={\bf N}_{\mathrm{Aut}(\mathrm{PSp}_n(q))}(\langle g\rangle)$. Moreover, ${\bf C}_{\mathrm{PSp}_n(q)}(g)\cong {}^2B_2(q)$. Since $|{}^2B_2(q)|=(q^2+1)q^2(q-1)$ we have that $|\mathrm{Aut}(\mathrm{PSp}_n(q)):{\bf C}_{\mathrm{PSp}_n(q)}(g)|$ is relatively prime to $t_1$.

Suppose that $T_1\nleq {\bf C}_{\mathrm{Aut}(\mathrm{PSp}_n(q))}(g)$. The usual argument using the faithful action of ${\bf N}_{\mathrm{Aut}(\mathrm{PSp}_n(q))}(\langle g\rangle)/{\bf C}_{\mathrm{Aut}(\mathrm{PSp}_n(q))}(g)$ on $\langle g\rangle$ gives that $g\in \mathrm{PSp}_n(q)$ and that $g$ has order divisible by a prime $p$ with $t_1\mid p-1$. However, as $t_1>fn\geqslant n$, we see that $T_1$ cannot permute non-trivially the eigenspaces of $g$.
\end{proof}

\begin{lemma}\label{hupPSp}Let $n\geqslant 6$ be even. Suppose  $r^{f(n-2)}-1$ admits a primitive prime divisor $t_2$. Let $g\in \mathrm{Aut}(\mathrm{PSp}_n(q))$ with $t_2$ dividing $|{\bf N}_{\mathrm{Aut}(\mathrm{PSp}_n(q))}{(\langle g\rangle)}|$ and let $T_2$ be a cyclic subgroup of order $t_2$ in 
${\bf N}_{\mathrm{Aut}(\mathrm{PSp}_n(q))}{(\langle g\rangle)}$. Then 
\[
g\in {\bf C}_{\mathrm{Aut}(\mathrm{PSp}_n(q))}(T_2)=(q^{\frac{n}{2}-1}+1)\circ\mathrm{GSp}_2(q).
\]In particular, $|\mathrm{Aut}(\mathrm{PSp}_n(q)):{\bf C}_{\mathrm{Aut}(\mathrm{PSp}_n(q))}(g)|$ is relatively prime to $t_2$.
\end{lemma}

\begin{proof} Using the information in~\cite[Chapter~4]{GLS}, we obtain
\[
{\bf C}_{\mathrm{Aut}(\mathrm{PSp}_n(q))}(T_2)=(q^{\frac{n}{2}-1}+1)\circ\mathrm{GSp}_2(q).
\]
Let $T$ be the cyclic subgroup of order $q^{n/2-1}+1$ in ${\bf C}_{\mathrm{Aut}(\mathrm{PSp}_n(q))}(T_2)$ and observe that $T$ is central in ${\bf C}_{\mathrm{Aut}(\mathrm{PSp}_n(q))}(T_2)$.

Suppose that $T_2\leqslant {\bf C}_{\mathrm{Aut}(\mathrm{PSp}_n(q))}(g)$. 
Then $g\in {\bf C}_{\mathrm{Aut}(\mathrm{PSp}_n(q))}(T_2)$ and hence $g$ centralizes $T$. Thus $T\leqslant {\bf C}_{\mathrm{Aut}(\mathrm{PSp}_n(q))}(g)$ and hence
$|\mathrm{Aut}(\mathrm{PSp}_n(q)):{\bf C}_{\mathrm{Aut}(\mathrm{PSL}_n(q))}(g)|$ 
is relatively prime to $t_2$, because so is $|\mathrm{PSp}_n(q):T|$.
 
Suppose that $T_2\nleq {\bf C}_{\mathrm{Aut}(\mathrm{PSp}_n(q))}(g)$. The usual argument using the faithful action of ${\bf N}_{\mathrm{Aut}(\mathrm{PSp}_n(q))}(\langle g\rangle)/{\bf C}_{\mathrm{Aut}(\mathrm{PSp}_n(q))}(g)$ on $\langle g\rangle$ gives that $g\in \mathrm{PSp}_n(q)$ and that $g$ has order divisible by prime $p$ with $t_2\mid p-1$.  Without loss of generality, we may suppose that $|g|=p$ and hence $g\in\mathrm{PSp}_n(q)$. Moreover, $g$ is semisimple, because $t_2$ divides $p-1$ and $t_2$ cannot divide $r-1$: recall that $t_2$ is a primitive prime divisor of $r^{f(n-2)}-1$. Now, there exists $\alpha\in\mathbb{N}$ with
$$t_2=\alpha f(n-2)+1.$$
If $\alpha f>1$, then $t_2>n$ and hence $T_2$ cannot permute non-trivially the eigenspaces of $g$. If $\alpha f=1$, then $t_2=n-1$ is prime. We show that this is impossible. Under the action of $T_2$, the vector space $V=\mathbb{F}_q^n$ decomposes as
$$V=W\perp W^\perp,$$
where $\dim_{\mathbb{F}_q}(W)=n-2$ and the symplectic form induced by $\mathrm{PSp}_n(q)$ on $W$ is non-degenerate, $W$ is an irreducible $\mathbb{F}_qT_2$-module, and $W^\perp$ is a $2$-dimensional trivial module for $T_2$. Since we are supposing that $T_2$ normalizes $\langle g\rangle$, $T_2$ permutes the eigenspaces of $g$. The orthogonal decomposition of $V$ above yields that $g$ has one eigenspace of dimension $n-2$, and then another eigenspace of dimension $2$, or two eigenspaces of dimension $1$. In either case, since $t_2=n-1$ is prime, we see that $T_2$ fixes setwise each eigenspace of $g$. From this it follows that $T_2$ centralizes $g$.
\end{proof}
\begin{lemma}\label{casePSp}
 If $L=\PSp_n(q)$, then $(G,x,y)$ is in Line~$10$  of Table~$\ref{table1}$.
\end{lemma}
\begin{proof}
When 
$$(n,q)\in \{(4,2),(4,3), (4,4),(4,8),(4,16),(4,32),(6,2),(6,3),(8,2)\},$$
the proof follows with a computation with the computer algebra system \textsc{Magma}. Now, by excluding these cases, the maximal factorizations of the almost simple groups with socle $L=\mathrm{PSp}_n(q)$ appear in~\cite[Table~1 and~2]{LPS1}. Moreover, by excluding these cases,  we simplify some of the computations later in the proof.

Let $t_1$ be a primitive prime divisor of $r^{nf}-1$: as usual the existence of $t_1$ follows from Zsigmondy's theorem. 
As $t_1$ divides $|L|$, without loss of generality, we may suppose that $t_1$ divides $|X|$. Let $T_1$ be a cyclic subgroup of $X$ of order $t_1$, let $C_1:={\bf C}_{\mathrm{Aut}(\mathrm{PSp}_n(q))}(T_1)$ and let $T_x$ be a maximal torus of $\mathrm{PGSp}_n(q)$ containing $T_1$. Thus $T_x$ is a torus of order $q^{n/2}+1$ and $T_x$ is a cyclic subgroup of $\mathrm{PGSp}_n(q)$. From Lemma~\ref{hup0PSp}, we obtain
\[
g\in C_1=
\begin{cases}
\langle T_x,\iota\rangle,&\textrm{when }n=4, r=2, f \textrm{ is odd, } \iota \textrm{ is a graph-field automorphism},\\
T_x,&\textrm{otherwise}.
\end{cases}
\]

\subsection*{Assume $x\in T_x$.} Thus $x$ is a semisimple element and $X$ is a field extension subgroup of $G$. We now discuss the structure of $X$. Assume first, for simplicity, that $x\in \mathrm{PSp}_n(q)$ or $n/2$ is odd. Thus, using~\cite[Section~$3.4$, Table~B.7]{BGbook} and~\cite[Chapter~4 and Table~$4.5.1$]{GLS}, we see that  the order of $X$ divides 
\begin{equation}\label{eq:c19uu}
f\ell|\mathrm{GU}_{n/\ell}(q^{\ell/2})|2,
\end{equation}
for some divisor $\ell$ of $n$ with $n/\ell$ odd.
When $x\in T_x\setminus\mathrm{PSp}_{n}(q)$ and $n/2$ is even, there are a few more cases to consider. As $T_x\nleq \mathrm{PSp}_{n}(q)$, $q$ is odd. As $n/2$ is even, $q^{n/2}+1\equiv 2\pmod 4$ and hence $T_x=\langle a\rangle\times\langle b\rangle$, where $a$ has order $2$ and $\langle b\rangle=T_x\cap\mathrm{PSp}_n(q)$ has order $(q^{n/2}+1)/2$. Using the references above, we obtain that $${\bf N}_{\mathrm{PGSp}_n(q)}(\langle a\rangle)={\bf C}_{\mathrm{PGSp}_n(q)}(a)\cong \mathrm{Sp}_{n/2}(q^2).2,$$
where the ``2'' on top acts as a field automorphism. Therefore,  when $q$ is odd and $n/2$ is even, we have the following possibilities for the order of $X$:
\begin{equation}\label{eq:c19uubis}
2f|\mathrm{Sp}_{n/2}(q^2)|\quad\textrm{and}\quad f\ell|\mathrm{GU}_{n/\ell}(q^{\ell/2})|,
\end{equation}
for each divisor $\ell$ of $n$ with $n/\ell$ odd.

 Before considering the element $y$ in general, we first consider the case $n=4$. In this case, from~\eqref{eq:c19uu} and~\eqref{eq:c19uubis}, $|X|$ divides either $4f(q^2+1)$ or 
$2f|\SL_2(q^2)|=2fq^2(q^4-1)$. In the first possibility, we have 
\begin{equation}\label{20200401_1}
m(\mathrm{PSp}_4(q))\leqslant |G:Y|\leqslant |X|\leqslant 4f(q^2+1),
\end{equation}
where $m(\mathrm{PSp}_4(q))$ is the minimal degree of a faithful permutation representation of $\mathrm{PSp}_4(q)$.
 Now, $m(\PSp_4(q))=(q^4-1)/(q-1)$, except when $q\in \{2,3\}$. (This information is tabulated, for instance, in~\cite[Table~4]{Guest}.) The 
inequality~\eqref{20200401_1} is satisfied only when $r=2$ and $f\leqslant 3$, or $q=r=3$. However, these cases have been checked with the help of a computer. 
Therefore, we may suppose that $\cent {\Aut(L)}x\cong \SL_2(q^2).2f$. In particular, $q$ is odd  and $x$ is an involution. Observe that $q^2\equiv 1\pmod 4$ and hence $(q^2+1)/2\equiv 1\pmod 2$. Therefore  $x\in \mathrm{PGSp}_4(q)\setminus \PSp_4(q)$ and hence
$\mathrm{PGSp}_4(q)\leqslant G$.  In particular, $|G:X|=q^2(q^2-1)/2$.

From~\cite[Table~8.12]{BHRDbook}, we see that $X=\cent G x$ is a maximal subgroup of $G$. From the classification of the maximal factorizations of almost simple groups~\cite{LPS1,LPS2}, we deduce that $Y$ is contained in a parabolic subgroup $P$ of $G$ whose unipotent radical $Q$ is a non-abelian group of order $q^3$ (this information is in~\cite[(3.2.1a)]{LPS1}). Furthermore, from~\cite[Table~8.12]{BHRDbook}, we get that the shape of $P\cap L$ is $$E_q^{1+2}:((q-1)\circ \mathrm{Sp}_2(q)).$$ As $G=XY$, we deduce that $|G:X|=q^2(q^2-1)/2$ divides $|Y|$. 
Using this description of $P$ it is not hard to see that the only elements $y\in P$ with the property that $\nor P{\langle y\rangle}$ has order divisible by $q^2(q^2-1)/2$ are the elements in $\Z Q\cong E_q$. Therefore, $y$ is a transvection of $\PSp_4(q)$. In particular, we find the examples in Line 10 of Table~\ref{table1} (for $n=4$).

\smallskip

Suppose now that $n\geqslant 6$. Let  $t_2$ be a primitive prime divisor of $r^{f(n-2)}-1$. Observe that $t_2$ does exist because $n\geqslant 6$ and because we are excluding the case $(n,q)=(8,2)$ from our analysis here. From~\eqref{eq:c19uu} and~\eqref{eq:c19uubis}, $|X|$ is relatively prime to $t_2$ and hence $t_2$ divides $|Y|$. Let $T_2$ be a cyclic subgroup of $Y$ of order  $t_2$ and set $C_2:=\cent {\Aut(L)} {T_2}$. From Lemma~\ref{hupPSp}, we obtain
\[
y\in C_2=(q^{\frac{n}{2}-1}+1)\circ\mathrm{GSp}_2(q).
\]
Write $$y=y_{n-2}y_2,$$ where $y_{n-2}$ belongs to the torus of cardinality $q^{n/2-1}+1$ and $y_2$ belongs to $\mathrm{GSp}_2(q)$. Now, this decomposition of $y$ induces a direct sum decomposition of the underlying vector space $V=V_{n-2}\perp V_2$, where $y$ induces $y_{n-2}$ on $V_{n-2}$ and induces $y_2$ on $V_2$. Since $y_{n-2}$ is semisimple, the $\langle y_{n-2}\rangle$-module $V_{n-2}$ is the direct sum of pair-wise isomorphic irreducible modules. 

\smallskip

\noindent\textsc{Suppose that none of these modules is isomorphic to  any of the irreducible $\langle y_2\rangle$-submodules of $V_2$} (here, we are including the possibility that $y_2$ is a non-identity unipotent element and hence $V_2$ is indecomposable with a unique irreducible submodule, which is the trivial module). 

In this case, $\nor G{\langle y\rangle}=Y$ preserves the direct sum decomposition $V_2\perp V_{n-2}$ and hence $Y$ is contained in the stabilizer in $G$ of a $2$-dimensional non-degenerate subspace of $V$. Now, by checking the maximal factorizations of the almost simple group $G$ in~\cite[Table~1 and~2]{LPS1}, we see that one of the following holds:
\begin{enumerate}[label=(\roman*)]
\item\label{sp:sp1} $X\cap L\leqslant \Sp_{n/2}(4).2$, $q=2$, $n/2$ is even, 
\item\label{sp:sp2} $X\cap L\leqslant \Sp_{n/2}(16).2$, $q=4$, $n/2$ is even, $G=\Aut(L)=L.2$,
\item\label{sp:sp3} $X\cap L\leqslant G_2(q)$, $q$ is even and $n=6$.
\end{enumerate}
All of these three cases can be eliminated with a computation. Indeed, 
since $\Sp_{n/2}(4).2$, $\Sp_{n/2}(16).2$ and $G_2(q)$ normalize no non-identity cyclic subgroup, we deduce that
$X\cap L$ must be strictly contained in this embedding. However, by comparing the order of $X$ (see~\eqref{eq:c19uu}), $Y$ and $G$, we see that the equality $G=XY$ cannot be satisfied. These computations can be performed in the same spirit as the analogous computations  for almost simple groups having socle $\mathrm{PSL}_n(q)$ and $\mathrm{PSU}_n(q)$. For instance, in Case~\ref{sp:sp3}, we have $n=6$ and hence, from~\eqref{eq:c19uu}, we have that $|X|$ divides $2f|\mathrm{GU}_3(q)|$; moreover, $|Y|$ divides $|\mathrm{Sp}_2(q)\perp\mathrm{Sp}_4(q)|f$. A computation shows that $|X|_2|Y|_2<|G|_2$, contradicting $G=XY$. We omit the details of the remaining computations for Cases~\ref{sp:sp1} and~\ref{sp:sp2}.

This means that, in order to have a factorization $G=XY={\bf N}_G(\langle x\rangle){\bf N}_G(\langle y\rangle)$, some of the $\langle y_{n-2}\rangle$-irreducible submodules of $V_{n-2}$ are isomorphic to some of the irreducible submodules of $V_2$. As $y_2\in \mathrm{GSp}_2(q)$, this implies that the irreducible $\langle y_{n-2}\rangle$-submodules of $V_{n-2}$ have dimension at most $2$ and hence $y_{n-2}$ has order a divisor of $q+1$.

\smallskip

\noindent\textsc{Suppose that $y=y_2$ is a unipotent element.} Thus $y$ is a transvection of $\PSp_n(q)$ and 
\[
\nor {\mathrm{PGSp}_n(q)}{\langle y\rangle}\cong
E_{q}^{1+(n-2)}:((q-1)\times \mathrm{Sp}_{n-2}(q)).\]
Assume that $X$ is of type $\Sp_{n/2}(q^2)$ (recall from~\eqref{eq:c19uu} and~\eqref{eq:c19uubis} that when this happens, $x$ is an involution and $n/2$ is even). Since $n/2$ is even, $q^{n/2}+1\equiv 2\pmod 4$ and hence $x\in \mathrm{PGSp}_n(q)\setminus\PSp_n(q)$. In particular, we find one of the examples in Line~10 of Table~\ref{table1}. When $X$ is not of type $\Sp_{n/2}(q^2)$, we deduce, by consulting the factorizations of the almost simple groups with socle $L=\PSp_n(q)$ in~\cite{LPS1} and by consulting the structure of $X$ and $Y$, that there are no triples $(G,x,y)$ occurring in this case.

\smallskip

\noindent\textsc{Suppose that $y$ is not a unipotent element.} If $r$ divides $|y|$, then $y_2$ is a non-identity unipotent element. However, this contradicts the fact that some of the $\langle y_{n-2}\rangle$-irreducible submodules of $V_{n-2}$ are isomorphic to some of the irreducible submodules of $V_2$. Therefore $r$ is relatively prime to $|y|$ and hence $y$ is a semisimple element. From the compatibility condition between the $\mathbb{F}_q\langle y_{n-2}\rangle$-submodules of $V_{n-2}$ and the  $\mathbb{F}_q\langle y_2\rangle$-submodules of $V_2$, we obtain $|y|=|y_{n-2}|=|y_2|$ and that $|y|$ is a divisor of $q+1$. This gives that $\cent L y\cong {\hat{}\,}\GU_{n/2}(q)$ or $\cent L y\cong \Sp_{n/2}(q^2)$, depending on whether $n/2$ is odd or even. Using this information on the structure of $X$ and $Y$ and consulting the list of maximal factorizations for $G$ in~\cite[Table~1 and~2]{LPS1}, we deduce that there are no examples arising in this case.

\subsection*{Assume that $x\notin T_x$.} This implies $n=4$, $r=2$ and $f$ is odd. Note that $|T_x|$ is odd while $|x|$ is even. Thus by Lemma \ref{lem:primeorder} we may assume that $x$ is an involution and so is a  
 graph-field automorphism of $\PSp_4(q)=\mathrm{Sp}_4(q)$.  Thus $X\geqslant {\bf N}_L(\langle x\rangle)=\cent L x\cong {^{2}B_2(q)}$ is a Suzuki group. From~\cite[Table~8.14]{BHRDbook}, we deduce that $X$ is a maximal subgroup of $G$. Using the classification of the maximal factorizations of the almost simple groups having socle $\mathrm{Sp}_4(q)$~\cite[Table~2]{LPS1}, we deduce that 
$$Y\cap L\leqslant \mathrm{O}_4^+(q)=\mathrm{SL}_2(q)\times \mathrm{SL}_2(q).$$ 

Suppose that $Y\cap L$ does not contain any of the two simple direct factors of $\mathrm{O}_4^+(q)$. Then $|\mathrm{O}_4^+(q):Y\cap L|\geqslant (q+1)^2$ because the minimal degree of a permutation representation of $\SL_2(q)$ is $q+1$. Therefore $|Y\cap L|\leqslant |\SL_2(q)|^2/(q+1)^2=q^2(q-1)^2$ and hence 
$$|G|\leqslant |X||Y|\leqslant 2f(q^2+1)q^2(q-1)\cdot q^2(q-1)^2|Y:L\cap Y|.$$
As $|G|=|G:L|q^4(q^4-1)(q^2-1)$ and $|G:L|\geqslant |Y:Y\cap L|$, we deduce $$(q+1)^2\leqslant 2f(q-1),$$ which is impossible. Therefore $Y$ contains at least one of the two simple direct factors of $\mathrm{O}_4^+(q)$. Let us denote by $S_1$ and $S_2$ the two simple direct factors of $\mathrm{O}_4^+(q)$. Without loss of generality we assume $S_1\leqslant Y$. Since $\nor G{\langle y\rangle}/\cent G y$ is soluble and since $\SL_2(q)$ is simple, we deduce $S_1\leqslant \cent G y$ and hence $y\in \cent G {S_1}$. Using the action of the outer automorphism group of $L=\mathrm{Sp}_4(q)$, we deduce $\cent G{S_1}\leqslant L$ and hence $y\in \cent L{S_1}$. As $\cent L{S_1}\leqslant \nor L {S_1}\leqslant \mathrm{O}_4^+(q)$, we have $\cent L{S_1}=S_2$ and hence $y\in S_2$. Since the normalizers of the non-identity elements of $S_2\cong\SL_2(q)$ have order $q$, $2(q-1)$ or $2(q+1)$, we deduce $$|L\cap Y|\leqslant |\mathrm{O}_4^+(q)\cap Y|=|\nor {\mathrm{O}_4^+(q)}{\langle y\rangle}|\leqslant  
|S_1|2(q+1)=2q(q^2-1)(q+1).$$ Now, as
$$|G:L|q^4(q^4-1)(q^2-1)=|G|\leqslant |X||Y|\leqslant 2f(q^2+1)q^2(q-1)\cdot 2q(q^2-1)(q+1)|Y:L\cap Y|$$
and $|G:L|\geqslant |Y:L\cap Y|$, we get $q\leqslant 2f$, which is a contradiction. Hence no triple arises in this case.

\smallskip

Using the explicit description of ${\bf N}_G(\langle x\rangle )$, ${\bf N}_G(\langle y\rangle)$ in Line~10, it is readily seen that $G={\bf C}_G(x){\bf C}_G(y)$.  Thus we have the $\surd$ symbol in Line~10.
\end{proof}

\section{Classical groups: odd dimensional orthogonal groups}\label{sec:classicalodd}

The analysis in this section is similar to the work in Section~\ref{sec:classicalpsp}; indeed, from one side, we use the classification of Liebeck, Praeger and Saxl~\cite{LPS1,LPS2} and, from the other side, the factorizations arising in the context of odd dimensional orthogonal groups resemble the factorizations for symplectic groups.
\begin{lemma}\label{casePO}
 If $L=\POmega_n(q)$ with $n$ odd, then $(G,x,y)$ is in Line~$14$ of Table~$\ref{table1}$.
\end{lemma}
\begin{proof}
When $(n,q)=(7,3)$, the proof follows with a computer computation: no example arises. Thus we assume that $(n,q)\neq (7,3)$. Set $m:=(n-1)/2$. We start by summarizing the maximal factorizations $$G=AB$$ of almost simple groups with socle $L=\Omega_n(q)$ (as usual we use the notation from~\cite{LPS1,LPS2}):
\begin{enumerate}
\item\label{casee1} $A\cap L=N_1^-$ and $B\cap L=P_m$,
\item\label{casee2} $n=7$, $A\cap L=G_2(q)$ and $B\cap L$ is either $P_1$, or $N_1^\varepsilon$, or
$N_2^{\varepsilon}$, with $\varepsilon\in \{+,-\}$,
\item\label{casee3} $n=13$, $q=3^f$, $A\cap L=\PSp_6(3^f).a$ with $a\leqslant 2$ and $B\cap L=N_1^-$, 
\item\label{casee4} $n=25$, $q=3^f$, $A\cap L=F_4(3^f)$ and $B\cap L=N_1^-$.
\end{enumerate}
Replacing $X$ by $Y$ if necessary, we may suppose that 
\begin{equation}\label{eq:0987}X\leqslant A\hbox{ and }Y\leqslant B.
\end{equation}

\smallskip

\noindent\textsc{Cases~\eqref{casee2},~\eqref{casee3} and~\eqref{casee4}}

\smallskip

\noindent Here $A$ is an almost simple subgroup of $G$ and hence  $X$ is a core-free proper subgroup of $A$. By Lemma~\ref{l: smaller}, the factorization $G=XY$ gives rise to the factorization 
\begin{equation}\label{eq:anewanew}
A=X(Y\cap A)
\end{equation} of $A$. 

Now, from~\cite[Table~5]{LPS1}, we see that $F_4(3^f)$ admits no proper factorizations. Hence $A\leqslant X$ or $A\leqslant Y\cap A$, which are both impossible.  Therefore Case~\eqref{casee4} does not arise. 

Assume Case~\eqref{casee2}. From~\cite[Table~5]{LPS1}, we see that $G_2(q)$ admits proper factorizations with $q$ odd only when $q=3^{f}$. Observe that $f\geqslant 2$, because we have dealt with $\Omega_7(3)$ above. Let $A=A'B'$ be a maximal factorization of $A$ with $X\leqslant A'$ and with $Y\cap A\leqslant B'$. Using the maximal factorizations of $G_2(q)$, we see that $A'\cap G_2(q)$ is one of the following groups $$\mathrm{SL}_3(q), \,\mathrm{SL}_3(q).2,\, \mathrm{SU}_3(q),\, \mathrm{SU}_3(q).2,\, ^2G_2(q),$$
where, for the last case, we require $f$ odd, but we do not need this information here. From $G=AB=XB$ and~\eqref{eq:anewanew}, we deduce $|G:B|=|A:A\cap B|=|X:X\cap B|$ and hence $|G:B|$ divides $|X|$. Thus 
 \begin{equation}\label{eq:anewanewanew}
 |G:B| \hbox{ divides }|A'|.
 \end{equation}
Now,  
\[|G:B|=
\begin{cases}
q^3\frac{q^3-1}{2}&\textrm{when }L\cap B=N_1^-,\\
q^3\frac{q^3+1}{2}&\textrm{when }L\cap B=N_1^+,\\
q^5\frac{q^6-1}{2(q-1)}&\textrm{when }L\cap B=N_2^+,\\
q^5\frac{q^6-1}{2(q+1)}&\textrm{when }L\cap B=N_2^-,\\
\frac{q^6-1}{q-1}&\textrm{when }L\cap B=P_1.
\end{cases}
\] 
Using this explicit value of $|G:B|$ and using~\eqref{eq:anewanewanew}, we deduce that
\begin{itemize}
\item either $L\cap B=N_1^-$ and $A'\cap G_2(q)\in
 \{\SL_3(q),\SL_3(q).2\}$, or
 \item $L\cap B=N_1^+$ and $A'\cap G_2(q)\in \{\SU_3(q),\SU_3(q).2\}$.
 \end{itemize} As $q=3^f$, we have $\gcd(3,q-1)=\gcd(3,q+1)=1$ and hence $A'$ is an almost simple group with socle $\SL_3(q)$ or $\SU_3(q)$. 
 Recall now that, since $X\leqslant A'$, we have $X=\nor G{\langle x\rangle}=\nor {A'}{\langle x\rangle}$. Now, it is not hard to verify that $\Aut(\SL_3(q))$ and $\Aut(\SU_3(q))$ do not contain a non-identity cyclic subgroup $\langle x\rangle$ whose normalizer has order divisible by $|G:B|\in \{(q^3-1)q^3/2,(q^3+1)q^3/2\}$. Hence Case~\eqref{casee2} does not arise.

The analysis for Case~\eqref{casee3} is similar to Case~\eqref{casee2}, but simpler. The factorization for $\Omega_{13}(3^f)$ arising in Case~\eqref{casee3} is described in detail in~\cite[4.6.3, Lemma~A]{LPS1}.
Observe that $A$ is an almost simple group with socle $\PSp_6(q)$. As $\nor G{\langle x\rangle}=X\leqslant A$, we have $X=\nor A{\langle x\rangle}$. As $G=XB$ and $G=LB$, we deduce that 
$$|L:L\cap B|=|G:B|=|X:X\cap B|.$$ Since $L\cap B=N_1^-$ in Case~\eqref{casee3}, we have  $$|L:L\cap B|=|\Omega_{13}(q):N_1^-|=\frac{(q^6-1)q^6}{2}$$ and hence $(q^6-1)q^6/2$ divides $|X|$. Now, it is not hard to verify, using~\cite[Tables~8.28 and~8.29]{BHRDbook} that $\Aut(\PSp_6(q))$ contains no non-identity group elements $g\ne 1$ with  $|{\bf N}_{\mathrm{Aut}(\mathrm{PSp}_n(q))}(\langle g\rangle)|$ divisible by $q^6(q^6-1)/2$. Hence Case~\eqref{casee3} does not arise.

\smallskip

\noindent\textsc{Case~\eqref{casee1}.}

\smallskip

\noindent Replacing $X$ with $Y$ if necessary, $X\cap L\leqslant N_1^{-}$ and $Y\cap L\leqslant P_{m}$, where $N_1^{-}$ is the stabilizer in $L$ of a $1$-dimensional non-degenerate subspace of ``minus type'' and $P_{m}$ is the stabilizer in $L$ of a totally isotropic subspace of dimension $m$; in particular, $P_m$ is a parabolic subgroup. 

For simplicity, let $A_X$ be the stabilizer in $G$ of a $1$-dimensional non-degenerate subspace of ``minus type'' with $X\leqslant A_X$ and let $B_Y$ be a parabolic subgroup of $G$ with $Y\leqslant B_Y$. It will also be convenient to let $\hat{P}_m$ be a parabolic subgroup of $\mathrm{SO}_n(q)$ with $B_Y\cap L\leqslant \hat{P}_m$.  Thus $$A_X\cap L\cong \Omega_{2m}^-(q).2\,\,\textrm{  and  }\,\,B_Y\cap L\cong E_q^{\frac{m(m-1)}{2}+m}:{\scriptstyle\frac{1}{2}}\mathrm{GL}_{m}(q).$$ Moreover, 
\begin{equation}\label{eq:23456}
|G:A_X|=|L:L\cap A_X|=|\Omega_{2m+1}(q):\mathrm{P}\Omega_{2m}^-(q).2|=\frac{q^m(q^m-1)}{2}.
\end{equation}

Now $G=XP_m$ so $X$ act transitively on the set of all totally isotropic subspaces of dimension $m$. Since $(m,q)\neq (7,3)$ we have from \cite[Theorem 7.1]{GGP} that $\Omega_{2m}^-(q)\unlhd X$. Since $\Omega_{2m}^-(q)$ has trivial centre and is insoluble it follows that $\Omega_{2m}^-(q)\leqslant {\bf C}_G(x)$. Thus $\SO_n(q)\leqslant G$ and $x\in \SO_n(q)\backslash \Omega_n(q)$ is an involution.

We now fix an $\mathbb{F}_q$-basis $e_1,\ldots,e_m,w,f_1,\ldots,f_m$ of $V$ such that the symmetric matrix defining $L=\Omega_{2m+1}(q)$ with respect to this ordered basis  is the matrix
\[
J=
\begin{pmatrix}
0&0&I\\
0&1&0\\
I&0&0
\end{pmatrix},
\]
where we use $I$ to denote the $m\times m$ identity matrix. Using this matrix representation, $\hat{P}_m$ has unipotent radical subgroup
\[
Q=\left\{
\begin{pmatrix}
I&v&B\\
0&1&-v^t\\
0&0&I
\end{pmatrix}\mid v\in \mathbb{F}_q^m, B\in\mathrm{Mat}_{m\times m}(\mathbb{F}_q), B+B^t+vv^t=0
\right\}\cong E_q^{\frac{m(m-1)}{2}+m}.
\]
Moreover, the Levi complement of $Q$ in $\hat{P}_m$ is
\[
\mathcal{L}=\left\{
\begin{pmatrix}
A&0&0\\
0&1&0\\
0&0&(A^{-1})^t
\end{pmatrix}\mid A\in\mathrm{GL}_m(q)
\right\}\cong \mathrm{GL}_m(q).
\]
In what follows we also need the following subgroup
\[
Z=\left\{
\begin{pmatrix}
I&0&B\\
0&1&0\\
0&0&I
\end{pmatrix}\mid B\in\mathrm{Mat}_{m\times m}(\mathbb{F}_q), B^t=-B
\right\}\cong E_q^{\frac{m(m-1)}{2}}.
\]
A simple computation yields that $Z\leqslant \Z Q$.

Let $t$ be a primitive prime divisor of $r^{fm}-1$ and observe that the existence of $t$ is guaranteed by Zsigmondy's theorem, because $r$ is odd and $m\geqslant 3$. By~\eqref{eq:23456}, $t$ divides $|G:A_X|$ and hence $t$ divides $|Y|$. Let $T$ be a cyclic subgroup of $Y$ having order $t$. We claim that 
\begin{equation}\label{eq:98876}
y\in \cent G T.
\end{equation} 
Since $t$ divides $|Y|$ and since $t$ is a primitive prime divisor of $r^{fm}-1$, we deduce that $T\leqslant L$ and hence $T\leqslant L\cap B_Y\cong P_m$. Suppose that $t$ divides $|{\bf N}_G(\langle y\rangle):{\bf C}_G(y)|$. Then, from the faithful action of $T$ on $\langle y\rangle$, we deduce that $|y|$ is divisible by a prime number $p$ with $t\mid p-1$. Set $y':=y^{|y|/p}$. Now, $y'$ is an element of prime order $p$. Moreover, using the fact that $t$ is a primitive prime divisor of $r^{fm}-1$ and that $t\mid (p-1)$, we have $y'\in L\cap Y\leqslant P_m$ and $y'$ is semisimple. Furthermore, $t$ divides $|{\bf N}_G(\langle y'\rangle):{\bf C}_G(y')|$. In particular, $T$ and $y'$ are semisimple elements in $\hat{P}_m$ and hence, using the explicit description of $\hat{P}_m$ above, we deduce that $T$ and $y'$ are both in a Levi complement (which is isomorphic to $\mathrm{GL}_m(q)$) of $\hat{P}_m$. Applying Lemma~\ref{hup0}  to $\mathrm{GL}_m(q)$ yields that it is impossible to have $t$ divides $|{\bf N}_{\mathrm{GL}_m(q)}(\langle y'\rangle):{\bf C}_{\mathrm{GL}_m(q)}(y')|$. Since $t$ divides $|Y|=|{\bf N}_G(\langle y\rangle)|$, we get $T\leqslant {\bf C}_G(g)$. Therefore,~\eqref{eq:98876} holds true.

From~\eqref{eq:98876}, we deduce $y\in \mathrm{SO}_n(q)$ and hence 
\begin{equation}\label{eq:98877}
y\in \hat{P}_m.
\end{equation}
Let 
\[
m_t=
\begin{pmatrix}
\lambda &0&0\\
0&1&0\\
0&0&(\lambda^{-1})^t
\end{pmatrix}
\]
be a generator of $T$.  
Now, using the explicit description of $Q$ and $\mathcal{L}$ above, we see that $\cent{\hat{P}_m}T=T_m\ltimes W$, where $T_m$ is a torus in $\mathcal{L}\cong \mathrm{GL}_m(q)$ of cardinality $q^m-1$ and
\[
W=\left\{
\begin{pmatrix}
I&0&B\\
0&1&0\\
0&0&I
\end{pmatrix}\mid B\in\mathrm{Mat}_{m\times m}(\mathbb{F}_q), B+B^t=0, \lambda B\lambda^t=B
\right\}.
\]
From~\eqref{eq:98876} and~\eqref{eq:98877}, we have $y\in T_m\ltimes W$.

We claim that 
\begin{equation}\label{eq55555}
y\in W.
\end{equation}
We argue by contradiction and we suppose that $y\notin W$. Then, replacing $y$ by a suitable power, we may assume that $y$ has prime order and $y\in T_m$. Using the explicit description of $\hat{P}_m$, it can be deduced that
\[
Y\cap \mathrm{SO}_n(q)=\nor {\mathrm{SO}_n(q)}{\langle y\rangle}=\nor {\hat{P}_m}{\langle y\rangle}\subseteq Z\mathcal{L}.
\]
However, $Z\mathcal{L}$ is not transitive on the non-degenerate $1$-dimensional subspaces of ``minus type'' and hence $A_XY\ne G$, which is a contradiction because in Case~\eqref{casee1} $Y$ does act transitively on the set of $1$-dimensional non-degenerate subspaces of ``minus type''. Therefore, we must have $y\in W$ and $y$ is a unipotent element of order $r$. Therefore,~\eqref{eq55555} holds true.

 Now, it can be shown  that the set 
$$\{B\in \mathrm{Mat}_{m\times m}(q)\mid B^t=-B, \lambda B\lambda^t=B\}$$
contains a non-zero matrix only when $m$ is even. Thus $m$ is even. Moreover, from~\cite[Table~B.12]{BGbook}, we have
\begin{equation}\label{eqLy}\cent {L}y\cong E_q^{\frac{m(m-1)}{2}+m}: \mathrm{Sp}_m(q).
\end{equation}
Thus we obtain the examples in Line~14 of Table~\ref{table1}.

%
%

\smallskip 

Using the explicit description of ${\bf N}_G(\langle x\rangle )$, ${\bf N}_G(\langle y\rangle)$ in Line~14, it is readily seen that $G={\bf C}_G(x){\bf C}_G(y)$.  Thus we have the $\surd$ symbol in Line~14.
\end{proof}

\section{Classical groups: even dimensional orthogonal groups having Witt defect $1$}\label{sec:classicaloeven-}

We begin with the following lemma.

\begin{lemma}\label{casePo--}
Let $m\geqslant 5$ be odd, let $n:=2m$ and let $g\in \mathrm{Aut}(\mathrm{P}\Omega_n^-(q))$ with $g\ne 1$ and with $|{\bf N}_{\mathrm{Aut}(\mathrm{P}\Omega_n^-(q))}(\langle g\rangle)|$ divisible by $$
q^{\frac{m(m-1)}{2}}(q^{m-1}+1)(q^{m-2}-1)\cdots (q^2+1)(q-1)
.$$ Then $g$ is an involution not in $\mathrm{P}\Omega_n^-(q)$ and
\[
{\bf C}_{\mathrm{P}\Omega_n^-(q)}(g)\cong
\begin{cases}
\mathrm{Sp}_{n-2}(q)&\textrm{when }q \textrm{ is even},\\
\Omega_{n-1}(q)&\textrm{when }q \textrm{ is odd}.\\
\end{cases}
\]
\end{lemma}
\begin{proof}
Set $v:=q^{m(m-1)/2}(q^{m-1}+1)(q^{m-2}-1)\cdots (q^2+1)(q-1)$, $L:=\mathrm{P}\Omega_n^-(q)$ and $A:=\mathrm{Aut}(\mathrm{P}\Omega_n^-(q))$. The proof follows by an inspection of Section~3.5 and Tables~B.11,~B.12 in~\cite{BGbook}. We give some details to make this inspection more elementary.

Suppose first that $g$ has prime order. Assume also that $g\in L$. Now, 
\begin{align}\label{thursday}
|{\bf N}_A(\langle g\rangle):{\bf C}_L(g)|&=|{\bf N}_A(\langle g\rangle):{\bf N}_L(\langle g\rangle)||{\bf N}_L(\langle g\rangle):{\bf C}_L(g)|.
\end{align}
The first factor on the right hand side of~\eqref{thursday} divides $|\mathrm{Out}(L)|$. Observe that the second factor on the right hand side of~\eqref{thursday} divides $r-1$ when $g$ is unipotent (because $|g|=r$ and $\varphi(r)=r-1$) and divides $n$ when $g$ is semisimple (because ${\bf N}_L(\langle g\rangle)/{\bf C}_L(g)$ acts by permuting the eigenspaces of $g$). Therefore $|{\bf C}_L(g)|$ is divisible by $v/\ell$, where $\ell:=\ell_1\ell_2$, $\ell_1:=\gcd(v,|\mathrm{Out}(L)|)\leqslant 8f$ and $\ell_2\leqslant r-1$ when $g$ is unipotent and $\ell_2\leqslant 2m$ when $g$ is semisimple. Now, using~\cite[Section~3.5]{BGbook}, a case-by-case analysis shows that there is no $g$ having centralizer divisible by such a large number.

 Assume that $g\notin L$.
Let $h\in L\cap {\bf N}_A(\langle g\rangle)$. Then, $g^h=g^i$, for some $1\leqslant i\leqslant |g|-1$. Now, $$h^{-1}ghg^{-1}=g^{i-1}\in L\cap \langle g\rangle=1$$ and hence $i=1$. This shows that $L\cap {\bf N}_A(\langle g\rangle)={\bf C}_L(g)$. Therefore
$$|{\bf N}_A(\langle g\rangle):{\bf C}_L(g)|=|{\bf N}_A(\langle g\rangle)L:L|$$
and hence $|{\bf N}_A(\langle g\rangle):{\bf C}_L(g)|$ divides $|\mathrm{Out}(L)|$.  Therefore $|{\bf C}_L(g)|$ is divisible by $v/\ell$, where $\ell:=\gcd(v,|\mathrm{Out}(L)|)\leqslant 8f$. Now, using~\cite[Section~3.5]{BGbook} and the notation therein, we see that the only elements having prime order with $g\notin L$ and having centralizer divisible by such a large number are conjugate to $\gamma_1$ when $q$ is odd and to $b_1$ when $q$ is even. Moreover, the structure of ${\bf C}_L(g)$ is discussed in~\cite[Section~3.5.2]{BGbook} when $q$ is odd and in~\cite[Section~3.5.4]{BGbook} when $q$ is even. The proof of the lemma follows in this case. 

Suppose now that $g$ does not have prime order. We need to show that no extra case arises. Observe that from the previous part of the proof, $g$ has order a power of $2$. Without loss of generality, replacing $g$ by $g^{|g|/4}$ if necessary, we may suppose that $g$ has order $4$. Observe that $g^2$ is $A$-conjugate to $\gamma_1$ when $q$ is odd and to $b_1$ when $q$ is even. Set $\bar{A}:={\bf C}_A(g^2)/\langle g^2\rangle$ and adopt the ``bar'' notation for the projection of ${\bf C}_A(g^2)$ onto $\bar{A}$. We have
\[
\bar{A}\cong
\begin{cases}
\mathrm{Aut}(\mathrm{Sp}_{n-2}(q))&\textrm{when }q \textrm{ is even},\\
\mathrm{Aut}(\Omega_{n-1}(q))&\textrm{when }q \textrm{ is odd}.\\
\end{cases}
\]
Moreover, $\overline{{\bf N}_A(\langle g\rangle)}={\bf C}_{\bar{A}}(\bar{g})$. This shows that ${\bf C}_{\bar{A}}(\bar{g})$ has order divisible by $v/2$. When $q$ is odd, we may apply~\cite[Section~3.5]{BGbook} to the odd dimensional orthogonal group $\Omega_{n-1}(q)$ and we see that $\mathrm{Aut}(\Omega_{n-1}(q))$ contains no involutions whose centralizer has order divisible by $v/2$. Similarly,  when $q$ is even, we may apply~\cite[Section~3.4]{BGbook} to the symplectic group $\mathrm{Sp}_{n-2}(q)$ and we see that $\mathrm{Aut}(\mathrm{Sp}_{n-2}(q))$ contains no involutions whose centralizer has order divisible by $v/2$ (to check this it is useful to recall that $n-2=2m-2\geqslant 8$).
\end{proof}

We are also going to need \cite[Lemma 4.4]{LPSregsubs} that lists all possibilities of $\GammaO^-_{2m}(q)$ that act transitively on an orbit of $\Omega_{2m}(q)$ on nonsingular 1-subspaces. However, it is not claimed there that all groups listed are actually transitive. We rule out two possibilities with the following lemma.

\begin{lemma}\label{lem:ruleout}
 Let $Y\leqslant\GammaO_{2m}^-(2)$ such that $m\equiv 2\pmod 4$ and either  $\SU_{m/4}(2^4)$ or $\Omega_{m/2}^-(2^4)$ is normal in $Y$.  Then $Y$ does not act transitively on the set of nonsingular 1-subspaces. 
\end{lemma}

\begin{proof}
Note that $\SU_{m/4}(2^4)\leqslant\Omega_{m/2}^-(2^4)\leqslant \GammaO_{2m}^-(2)$ and so it suffices to show that $Y:=N_{\GammaO_{2m}^-(2)}(\Omega_{m/2}^-(2^4))$ is not transitive on the set of nonsingular 1-spaces.   
Let $k=\GF(2)$, $V=k^{2m}$ and $Q$ be a nondegenerate quadratic form on $V$ of ``minus type''. Let $\Delta=\{v\in V\mid Q(v)=1\}$, which corresponds to the set of all nonsingular 1-subspaces of $V$. Consider $V$ as a $m/2$-dimensional vector space over $K=\GF(2^4)$. Following \cite[p59]{LPS1}, let $P:V\rightarrow K$ be a nondegenerate quadratic form on $V$ of ``minus type'' such that $Q=\mathrm{Tr}_{K\rightarrow k} \circ P$, that is, $Q(v)=P(v)+P(v)^2+P(v)^{4}+P(v)^{8}$ for each $v\in V$. Note that $\Delta=\{v\in V\mid P(v)+P(v)^2+P(v)^{4}+P(v)^{8}=1\}$. Now arguing as in \cite[p59]{LPS1} we have that  $Y=\langle \Omega^-_{m/2}(2^4), \phi\rangle$ where $\phi:V\rightarrow V$ has order 8 and $P(v^\phi)=P(v)^\tau$ where $\tau$ is a generator of $\Aut(\GF(2^4))$. 

Let $v\in V$ such that $P(v)\neq 0$. Then $P(\langle v\rangle_K)=K$. Since $\mathrm{Tr}_{K\rightarrow k}$ is $k$-linear, its kernel has size $8$ and so there are precisely $8$ elements $w\in \langle v\rangle_K$ such that $Q(w)=1$. Since the isometry group of $P$ has index 4 in $Y$ it follows that $Y$ is not transitive on $\Delta$. 
\end{proof}

\begin{lemma}\label{casePO-}
 If $L=\POmega_n^-(q)$, then $(G,x,y)$ is in Line~$11$,~$12$ or~$13$ of Table~$\ref{table1}$.
\end{lemma}
\begin{proof}
When $(n,q)=(8,2)$, the proof follows with a computer computation with the computer algebra 
system \textsc{Magma}: there are no triples in this case. Set $m:=n/2$. From~\cite{LPS1,LPS2}, there exist two core-free maximal subgroups $A$ and $B$ of $G$, with $X\leqslant A$ and with $Y\leqslant B$. Moreover, replacing $x$ by $y$ if necessary, from~\cite{LPS1}, we see  that one of the following holds:-
\begin{enumerate}
\item\label{case1} $L=\POmega_{10}^-(2)$, $A\cap L=\Alt(12)$ and $B\cap L=P_1$,
\item\label{case2} $A\cap L=N_1$, $B\cap L= {\hat{}\,{\mathrm{GU}}}_m(q)$ and $m$ is odd,
\item\label{case3} $A\cap L=P_1$, $B\cap L= {\hat{}\,{\mathrm{GU}}}_m(q)$ and $m$ is odd,
\item\label{case4} $A\cap L=N_1$, $B\cap L= \Omega_m^-(q^2).2$, $q\in \{2,4\}$, $m$ is even and $G=\Aut(L)$,
\item\label{case5} $A\cap L=N_2^+$, $B\cap L= \mathrm{GU}_m(4)$, $q=4$, $m$ is odd and $G=\Aut(L)$.
\end{enumerate}

\smallskip 

\noindent\textsc{Case~\eqref{case1}. }

\noindent This case can be dealt with a computer computation and we obtain one of the examples in Line~11 of Table~\ref{table1}.  

\smallskip

\noindent\textsc{Case~\eqref{case5}. }

\noindent The factorization $G=XY$ of $G$ gives rise to the factorization $$B=G\cap B=XY\cap B=(X\cap B)Y$$ of $B$, via Lemma~\ref{l: smaller}. Let us denote by $g\mapsto \bar{g}$ the natural projection from $B$ to $\bar{B}=B/{\bf Z}(B\cap L)=B/\Z {\GU_m(4)}$ and observe that ${\bf Z}(\mathrm{GU}_m(q))$ has order $\gcd(m,q+1)=\gcd(m,5)$. Now, $\bar{B}$ is an almost simple group with socle $\PSU_m(4)$ with $m$ odd. An inspection of~\cite{LPS1} reveals that this group $\bar{B}$ has no maximal factorizations and hence the factorization $B=(X\cap B)Y$ implies $\bar{B}=\bar{Y}$, or $\bar{B}=\overline{X\cap B}$. The second option is absurd because, by hypothesis, $X\cap L\leqslant A\cap L=N_2^+$ and $N_2^+\cap B$ cannot project surjectively to $B$. Therefore $\bar{B}=\bar{Y}$ and hence $y\in {\Z {\mathrm{GU}_m(4)}}$. In particular, ${\Z {\mathrm{GU}_m(4)}}\ne 1$ and hence $5$ divides $m$. Moreover, 
\begin{equation}\label{eq:POmegaMinus}Y=\nor G {\Z {\mathrm{GU}_m(4)}},
\end{equation} $|y|=5$ and $y$ is a semisimple element having no eigenvalue in $\mathbb{F}_{q}$. 

Now,  by \cite[Lemma 4.1.1]{KL}, $\Omega_2^+(4)\times \Omega_{n-2}^-(4)\leqslant A\cap L=N_2^+=\mathrm{O}_2^+(4)\times \mathrm{O}_{n-2}^-(4)$ and $A\cap L \cap \mathrm{O}^+_2(4)=\Omega^+_2(4)$. Hence $A=\nor G{\Omega_2^+(4)}$ and $|\Omega_2^+(4)|=3$. If $\langle x\rangle=\Omega_2^+(4)$, we obtain the examples in Line~12 of Table~\ref{table1}. Suppose then $\langle x\rangle\ne \Omega_2^+(4)$. Let us denote by $g\mapsto \bar{g}$ the natural projection from $A$ to $\bar{A}=A/Z$. Now, $\bar{A}$ is an almost simple group with socle $\POmega_{2m-2}^-(4)$. As usual, from Lemma~\ref{l: smaller}, the factorization $G=XY$ gives rise to a factorization $A=X(A\cap Y)$ of $A$ and hence to the factorization  $\bar{A}=\bar{X}\overline{A\cap Y}$ of the almost simple group $\bar{A}$ having socle $\mathrm{P}\Omega_{2m-2}^-(q)$. As $(2m-2)/2=m-1$ is even, Case~\eqref{case4} holds for the factorization, $\bar{A}=\bar{X}\overline{A\cap Y}$, that is, $\bar{X}$ is contained in a subgroup of type $N_1$ of $\bar{A}$, or of type $\Omega_{m-1}^-(q^2).2=\Omega_{m-1}^-(16).2$. However, the first possibility is impossible, otherwise we would have a factorization of $G$ where one of the two factors (namely $X$) is contained in  the stabilizer of a $3$-dimensional non-degenerate subspace of $V=\mathbb{F}_q^n$, contradicting~\cite{LPS1}. In the second case, we claim that $|X|$ is not divisible by $|G:B|$. Indeed, from~\eqref{eq:POmegaMinus}, we have
$$|G:B|=|\mathrm{P}\Omega_n^-(q):\mathrm{GU}_m(q)|$$
is divisible by $q^{m-2}-1$, because $m$ is odd. However, if $t$ is a primitive prime divisor of $q^{m-2}-1$, then it is readily seen that $t$ is relatively prime to $|\Omega_{m-1}^-(q^2).2|$ and hence to $|X|$. 

\smallskip

\noindent\textsc{Cases~\eqref{case2}, \eqref{case3} and \eqref{case4}. }

\smallskip

\noindent Suppose first that $B\cap L={\hat{}}\GU_m(q))$. From the factorization $G=XB$, we deduce that $|X|$ is divisible by $|G:B|$ and hence by $$|LB:B|=|L:L\cap B|=q^{\frac{m(m-1)}{2}}(q^{m-1}+1)(q^{m-2}-1)\cdots (q^2+1)(q-1).$$
Now, Lemma~\ref{casePo--} yields that $x$ is an involution whose centraliser is $N_1$. Thus $A\cap L= N_1$.  Hence Case \eqref{case3} does not occur and we only need to consider Cases \eqref{case2} and \eqref{case4}. In both cases we have $A\cap L=N_1$. Hence $X\leqslant N_1$ and so $Y$ acts transitively on an $L$-orbit of nonsingular 1-subspaces. Thus by \cite[Lemma 4.4]{LPSregsubs} and Lemma \ref{lem:ruleout} we have that one of the following is a normal subgroup $Y_0$ of $Y$:
\begin{enumerate}
    \item $\SU_m(q)$ and $m$ odd;
    \item $\SU_{m/2}(q^2)$, with $m\equiv 2\pmod 4$, $m\geqslant 6$, and $q=2$ or 4;
    \item $\Omega_m^-(q^2)$ with $m$ even and $q=2$ or 4;
\end{enumerate}
The last possibility for $Y_0$ have trivial centre and is insoluble so must lie in $\cent{G}{y}$. However, inspecting \cite[Section 3.5]{BGbook} we see that this is not possible (note that $Y_0$ is irreducible).
When $Y_0=\SU_m(q)$ we have that $B=Y=\nor{G}{\langle y\rangle}$ for some $y\in \Z{B\cap L}=\Z{\,{\hat{}}\GU_m(q)}$. Thus $y$ is semisimple of order a divisor of $q+1$ and $y$ has no eigenvalue in $\mathbb{F}_{q}$. Moreover, the argument at the start of the paragraph yields that $x$ is an involution whose centraliser is $N_1$ and so we have Line 11.

It remains to consider the case where $Y_0=\SU_{m/2}(q^2)$ when $q=2$ or 4. Then   $Y=\nor{G}{\langle y\rangle}$ for some $y\in \Z{B\cap L}=\Z{\,{\hat{}}\GU_{m/2}(q^2)}$. Thus $y$ is semisimple of order $q^2+1$ and $y$ has no eigenvalue in $\mathbb{F}_{q^2}$. We also have that $B\cap L=\mathrm{O}_m^-(q^2)$ and so for the maximal factorisation to exist we need $G=\Aut(L)$.   Using the argument in \cite[p 59]{LPS1}, let $Q:V\rightarrow \GF(q)$ and $P:V\rightarrow \GF(q^2)$ be nondegenerate quadratic forms of ``minus type'' such that $Q=\mathrm{Tr}_{q^2\rightarrow q}\circ P$. Let $v\in V$ such that $Q(v)=1$. 
As the elements of $\langle v\rangle_{\GF(q^2)}$ have distinct $P$-values we have that $B\cap N_1= \mathrm{O}^-_{2m}(q^2)_v=\mathrm{O}^-_{2m}(q^2)_{\langle v\rangle_{\GF(q)^2}}=\Sp_{m-2}(q^2)\times C_2$. By \cite[Table 1]{LPS1} we have that $\mathrm{O}^-_m(q^2)=( \Sp_{m-2}(q^2)\times C_2)(\GU_{m/2}(q^2).2)$ and so $B=( \Sp_{m-2}(q^2)\times C_2)(\GU_{m/2}(q^2).4f)$, where $f=2$ if $q=4$ and $f=1$ otherwise. Thus $B=(B\cap N_1)Y$ and so by  Lemma \ref{l: smaller} $G=N_1Y$. Note that $N_1$ is the centraliser in $G$ of an involution in $\SO_{2m}^-(q)\backslash\Omega_{2m}^-(q)$  whose centraliser in $L$ is $\Sp_{n-2}(q)$, and so we have the factorisation in Line 13 of Table \ref{table1}. It remains to show that it is not possible to have $X<N_1$. Note that $N_1=\Sp_{n-2}(2)\times C_2$ when $q=2$, while when $q=4$ we have $N_1=(\Sp_{n-2}(4)\rtimes \langle\phi\rangle) \times C_2$ where $\phi$ is a field automorphism. Also if $x=(x_1,x_2)\in N_1$ then $X\leqslant N_{\Sp_{n-2}(4)\rtimes \langle\phi\rangle}(\langle x_1\rangle)\times C_2$. Note that it remains to consider the case where $x_1\neq 1$. Looking at $|X:B|$ we deduce that a primitive prime divisor of $r^{f(n-2)}-1$ divides $|X|$ and so by Lemma \ref{hup0PSp} we deduce that $x_1$ lies in a maximal torus of $\mathrm{PGSp}_{n-2}(q)$ of order $q^{(n-2)/2}+1$. Then looking at the possible orders of  $N_{\Sp_{n-2}(4)\rtimes \langle\phi\rangle}(\langle x_1\rangle)$ we deduce that no factorisation arises.

 \smallskip

In Line 11, we see from \cite[3.5.2(b)]{LPS1} that $ \hat{} \SU_m(q)\leqslant {\bf C}_G(y) $ acts transitively on an $L$-orbit of nonsingular 1-spaces and so we get $G={\bf C}_G(x){\bf C}_G(y)$.  However, for Lines 12 and 13 we see in \cite[3.5.1 and 3.5.2(c)]{LPS1} that ${\bf C}_G(y)$ needs to contain field automorphisms to be transitive on the conjugacy class $\langle x\rangle^G$. Since such elements of ${\bf N}_G(\langle y\rangle)$ do not centralise $y$ it follows that 
 ${\bf C}_G(x){\bf C}_G(y)<G$.  Thus we have the $\surd$ symbol in Line~11, whereas $\surd$ is omitted in Lines~12 and 13.
\end{proof}

\subsection*{Acknowledgements} The second author was supported by the Australian Research Council Discovery Project DP160102323.

\bibliographystyle{alpha}
\bibliography{szeprefs}

\end{document}